\documentclass[11pt]{article}
\usepackage{amsfonts,amsmath,amssymb,amsthm, amscd}
\usepackage{endnotes}
\usepackage{graphicx}
\usepackage{float}

\numberwithin{equation}{section}

\newtheorem{theorem}{Theorem}[section]
\newtheorem{prop}[theorem]{Proposition}

\newtheorem{cor}[theorem]{Corollary}

\theoremstyle{definition}
\newtheorem{definition}[theorem]{Definition}
\newtheorem{example}[theorem]{Example}
\newtheorem{remark}[theorem]{Remark}

\def\<{{\langle}}
\def\>{{\rangle}}
\def\G{{\Gamma}}

\def\S{\mathbb S}

\def\ni{\noindent}
\def\bs{\bigskip}

\def\N{\mathbb N}

\def\Si{{\Sigma}}
\def\-{\overline}

\makeatother

\let \ttorg \tt \def \tt{\ttorg \obeyspaces}

\begin{document}

\title{THE PENROSE-KAUFFMAN POLYNOMIAL}

\author{Louis H. Kauffman\\
Math, UIC\\
851 South Morgan Street\\ Chicago, Illinois 60607-7045\\ and \\ International Institute for Sustainability\\ with Knotted Chiral Meta Matter (WPI-SKCM$^2$)\\Hiroshima University\\1-3-1 Kagamiyama \\Higashi-Hiroshima, Hiroshima 739-8531 Japan \and Daniel S. Silver \\Department of Mathematics and Statistics\\ University of South Alabama \and Susan G. Williams\\Department of Mathematics and Statistics\\ University of South Alabama}

\maketitle  

\begin{abstract} For any cubic graph in a closed orientable surface and a perfect matching, the Penrose-Kauffman polynomial is a sum of  chromatic polynomials of a collection of associated graphs. A knot-theoretic perspective affords elementary proofs of old and new results about the polynomial. The Four Color Theorem is shown to be equivalent to a statement about 3-coloring alternating link diagrams in the plane that are reduced and have no bigon regions. \end{abstract} 

Keywords: Cubic graph, Penrose polynomial, link diagrams, alternating

MSC 2020:  
Primary 05C15; secondary 57M15.

\section{Introduction}  The Penrose polynomial of a plane cubic graph and perfect matching is implicit in \cite{Penrose}. Among its notable properties is that evaluation at 3 gives the number of proper edge 3-colorings of the graph. The polynomial has been studied in many places \cite{aigner, bm, em, ekm, jaeger}. 

In \cite{kauffman} the first author defined a generalized Penrose polynomial for cubic graphs in closed orientable surfaces, suggesting a knot-theoretic interpretation. Our purpose is to develop further that interpretation. We show how it makes possible elementary proofs of previously established theorems about the Penrose polynomial as well as proofs of new results. 

The first author's original version of the generalized Penrose polynomial is in the appendix. It can be read first if the reader so desires. 
\smallskip


\section{Tait Colorings}\label{TaitColorings} The Four Color Theorem asserts that every map in the plane can be colored using only four colors in such a way that any two adjacent regions receive distinct colors. It was originally proposed in 1852 by Francis Guthrie, a former student of Augustus de Morgan. A proof was given 1879 by Alfred Bay Kemp and another a year later by Peter Guthrie Tait. 
Both proofs turned out to be deficient. The theorem was finally proved with the use of computer in 1976 by Kenneth Appel and Wolfgang Haken. 

Tait's proof of the Four Color Theorem was incomplete, but he did establish the following. 

\begin{theorem}\label{tait} (P.G. Tait) The Four Color Theorem holds if and only if every bridgeless cubic plane graph has a proper edge 3-coloring. \end{theorem}

Recall that a {\it cubic graph} is a graph $G$ such that every vertex has degree 3. It is {\it plane} if it is embedded in $\S^2$. A {\it bridge} is an edge the removal of which disconnects $G$. Finally, a {\it proper edge 3-coloring} of $G$ is a function from the edge set of $G$ to a 3-element set $\{1, 2, 3\}$ so that at every vertex all three elements (``colors") appear. 

Any cubic graph $G$ has an even number vertices (an observation of Euler). A {\it perfect matching} of $G$ is a 
choice $M$ of edges of $G$ such that every vertex of $G$ is adjacent to exactly one edge of $M$. 

Given a cubic graph $G$ and perfect matching $M$, we follow \cite{kauffman} by considering assignments of colors $\{1, 2, \ldots, n\}$ to the edges of $G$ not in $M$, satisfying the condition that, for any edge $e \in M$, the set of colors $\{i, j\}$ of the two edges meeting one end is the same as the set of colors of the edges meeting the other, and moreover $i \ne j$.  (See Figure \ref{Matching}.) We will call such an assignment a  {\it  Tait $n$-coloring} of $(G,M)$. 

Note that a Tait $n$-coloring does not assign colors to the edges of the perfect matching $M$. Nevertheless every Tait 3-coloring extends uniquely to a proper edge 3-coloring of $G$. Conversely, every proper edge 3-coloring of $G$ restricts to a Tait 3-coloring of $G$ for any perfect matching $M$.

\begin{figure}[H]
\begin{center}
\includegraphics[height=1.2 in]{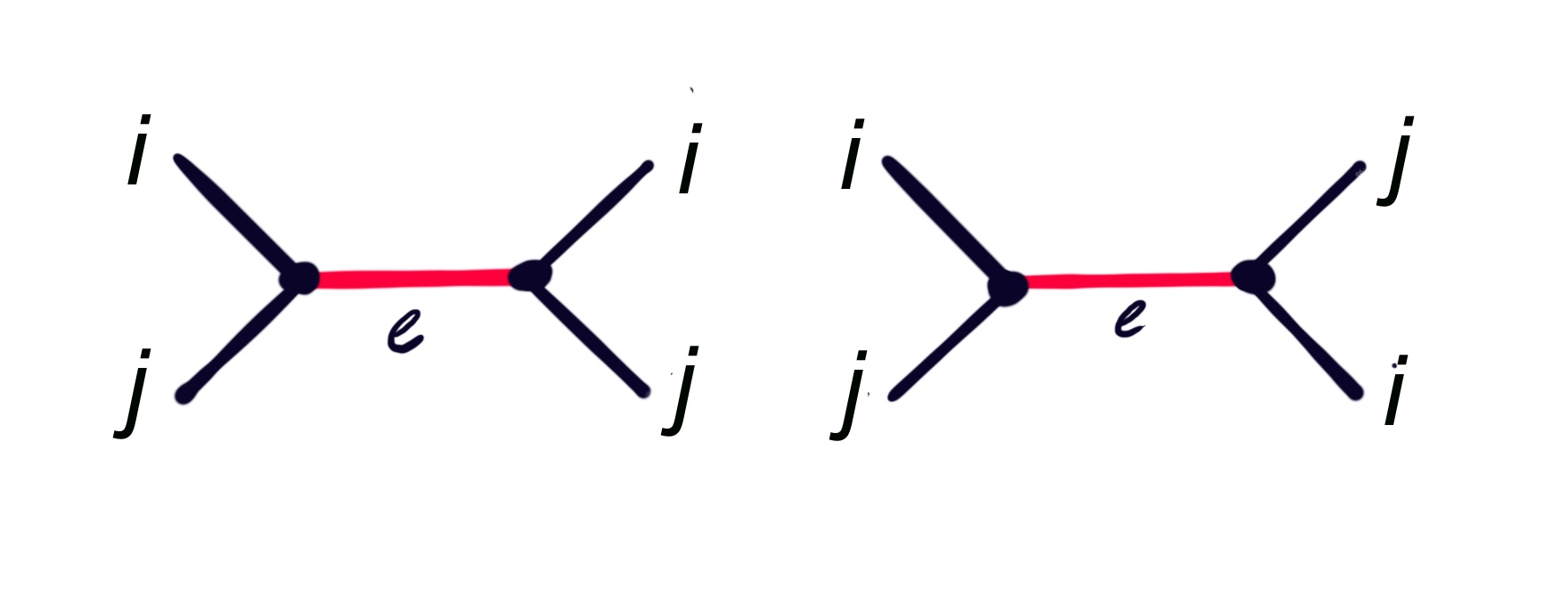}
\caption{Tait n-Coloring Condition}
\label{Matching}
\end{center}
\end{figure}


\section{Penrose-Kauffman Polynomial}  \label{PKPoly} All surfaces are assumed to be connected, closed and orientable.
In \cite{kauffman} the first author defined an integral polynomial for cubic graphs and perfect matchings in surfaces with the property that its evaluation at any positive integer $n \ge 3$ is the number of Tait $n$-colorings. Since there can be only one such polynomial, the following definition is equivalent to the original one. 

\begin{definition} Let $G$ be a cubic graph embedded in a surface together with a perfect matching $M$. The {\it Penrose-Kauffman (PK) polynomial} of $(G,M)$ is the unique integral polynomial $P(q)$ with the property that its evaluation at any positive integer $n$ is the number of Tait $n$-colorings of $(G,M)$. \end{definition}

An  expression for $P(q)$ as a sum of chromatic polynomials appears below in Theorem \ref{PK}.

In \cite{kauffman} the PK polynomial is defined via graph skein relations, examining states of the graph when modifications are performed on the edges of $M$ (see also Appendix). The contributions of states often cancel algebraically. Evaluating the remaining states requires care. As we show, one can recognize these states and their contributions to the PK polynomial. 

Consider any cubic graph $G$ together with a perfect matching $M$ embedded in a surface. Associate a link diagram $D = D(G,M)$ in the surface by replacing each edge $e \in M$ with a {\it crossing} as in Figure \ref{Knot},  a pair of arcs with an artistic device to indicate that if the surface is thickened then one arc ({\it over-arc}) would pass over the other  ({\it under-arc}).  A {\it link component} of $D$ is a closed curve gotten by traveling along the diagram. 

\begin{figure}[H]
\begin{center}
\includegraphics[height=.85 in]{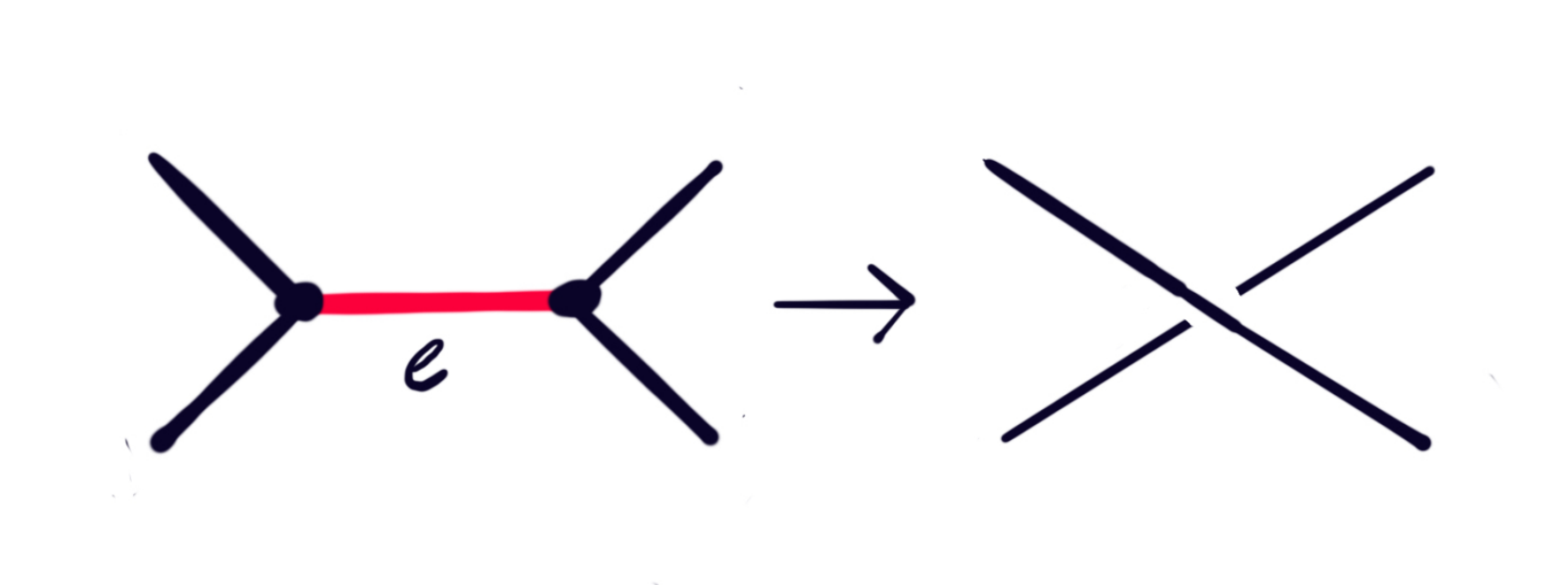}
\caption{Building a Link Diagram $D$ from $(G,M)$}
\label{Knot}
\end{center}
\end{figure}

A {\it (smoothing) state} of $D$ is the link diagram $S$ gotten by smoothing a subset of crossings in the manner shown in Figure \ref{2dots}.  It is helpful to imagine two dots near each crossing, placed so that if the over-arc sweeps counterclockwise, it passes the dots before the reaching the under-arcs. After smoothing, the dots should appear in the same region. (In fact, the dot convention eliminates any need for rendering over- and under-arcs with an artistic device.)  The diagram $D$ is itself a state while smoothing every crossing produces another, the {\it all-smoothed state} $S_0$. 

\begin{figure}[H]
\begin{center}
\includegraphics[height=1 in]{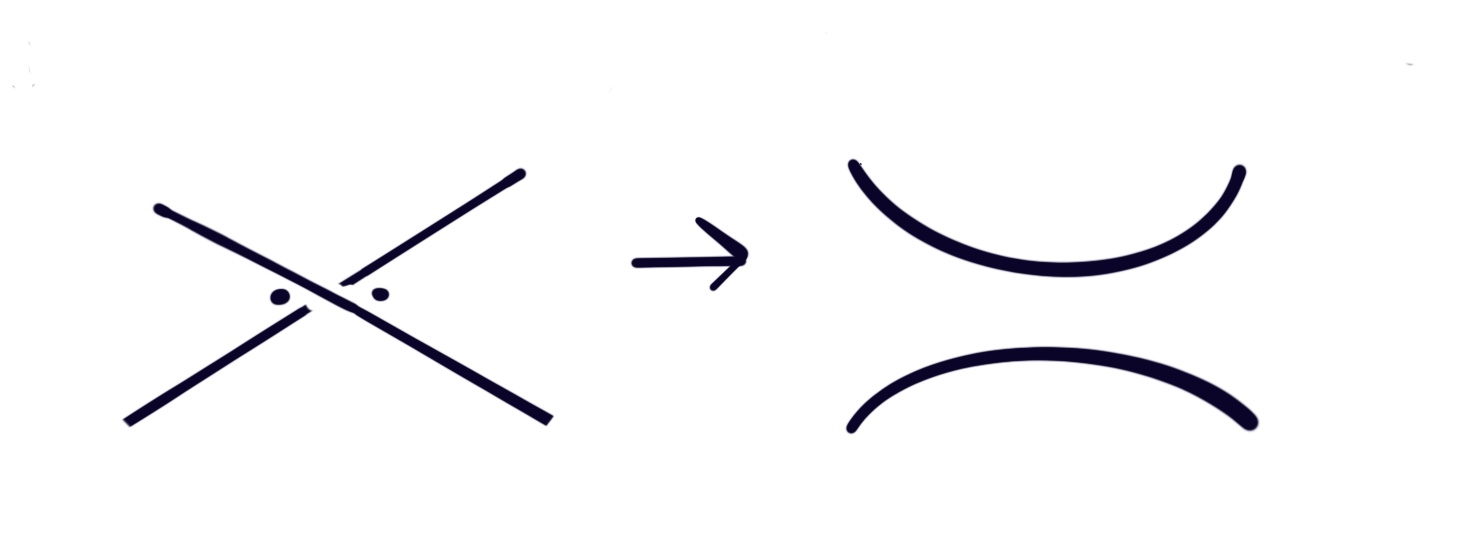}
\caption{Smoothing a Crossing of $D$}
\label{2dots}
\end{center}
\end{figure}

\begin{definition} \label{kiss} A state $S$ is {\it colorable} if, for each edge $e \in M$, the two arcs (crossing or parallel) that replace $e$ lie in different link components $C_e, C_e'$.  If the arcs cross, then we will say that the components {\it cross}; if they are parallel then the components {\it kiss}. \end{definition}

Note that a state is colorable if no component crosses or kisses itself. 

\begin{definition} \label{ncoloring} An {\it $n$-coloring} of a colorable state S is an assignment of colors $1, 2, \ldots, n$ to the link components of $S$ such any pair of components that cross or kiss receive different colors. \end{definition}

The number of $n$-colorings of a colorable state can be computed using the chromatic polynomial. (We will use well-known properties of the chromatic polynomial throughout. Details can be found in  \cite{biggs} and \cite{diestel}, for example.) First, we construct the {\it component graph} $G_S$, with vertices corresponding to the link components of $S$ and a single edge between vertices if the corresponding components cross or kiss. Let $\chi(G_S, q)$ be the chromatic polynomial of $G_S$. Then $\chi(G_S, n)$ is the number of $n$-colorings of  the state $S$. 

\begin{theorem} \label{PK} Let $G$ be a cubic graph with perfect matching $M$ embedded in a surface. The PK polynomial of $(G, M)$ is equal to
$$ \sum_S \chi(G_S, q),$$
where the summation is taken over all colorable states of the link diagram $D$.
\end{theorem}

\begin{proof} It suffices to establish a bijection between Tait $n$-colorings of $(G, M)$ and $n$-colorings of the colorable states of $D$. Figure \ref{Resolve} below shows how this is done. 
\end{proof} 

\begin{figure}[H]
\begin{center}
\includegraphics[height=1.8 in]{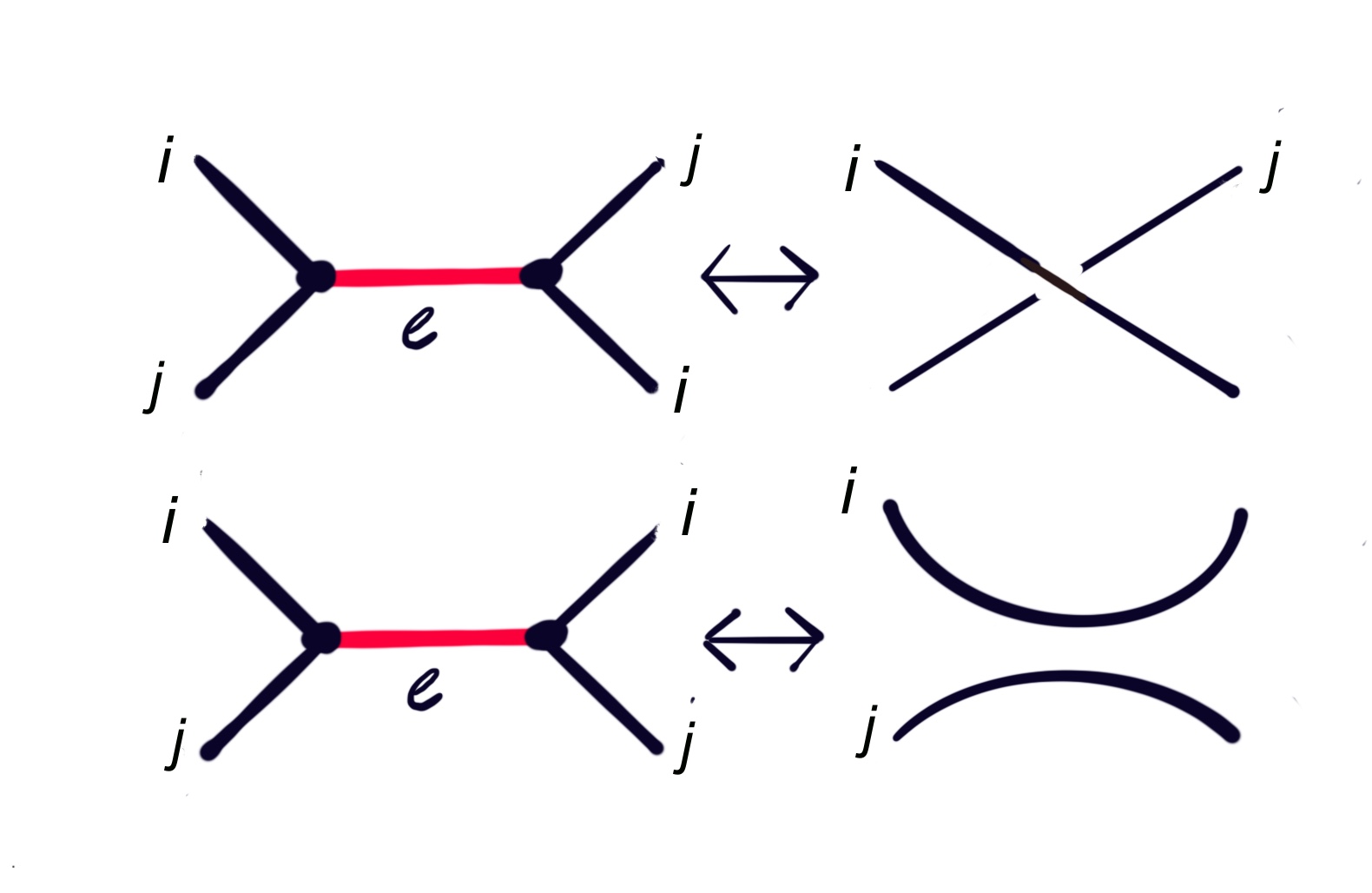}
\caption{Correspondence between  Tait $n$-colorings of $(G,M)$ and $n$-colorings of states $S$}
\label{Resolve}
\end{center}
\end{figure}

We will refer to the polynomial $P(q)$ as the {\it PK polynomial} of both $(G,M)$ and the associated link diagram $D$. If $D$ has no colorable states, then the polynomial is zero.

\begin{example}\label{trefoils}  A diagram of the left-hand trefoil knot appears in Figure~\ref{LHTrefoil}. (Plane link diagrams are regarded in the 2-sphere $\S^2$.) There is only one colorable state $S$, and it appears to the right with its associated component graph below. We have $P(q)= \chi(G_S, q) = q(q-1)(q-2)$. 

\begin{figure}[H]
\begin{center}
\includegraphics[height=1.8 in]{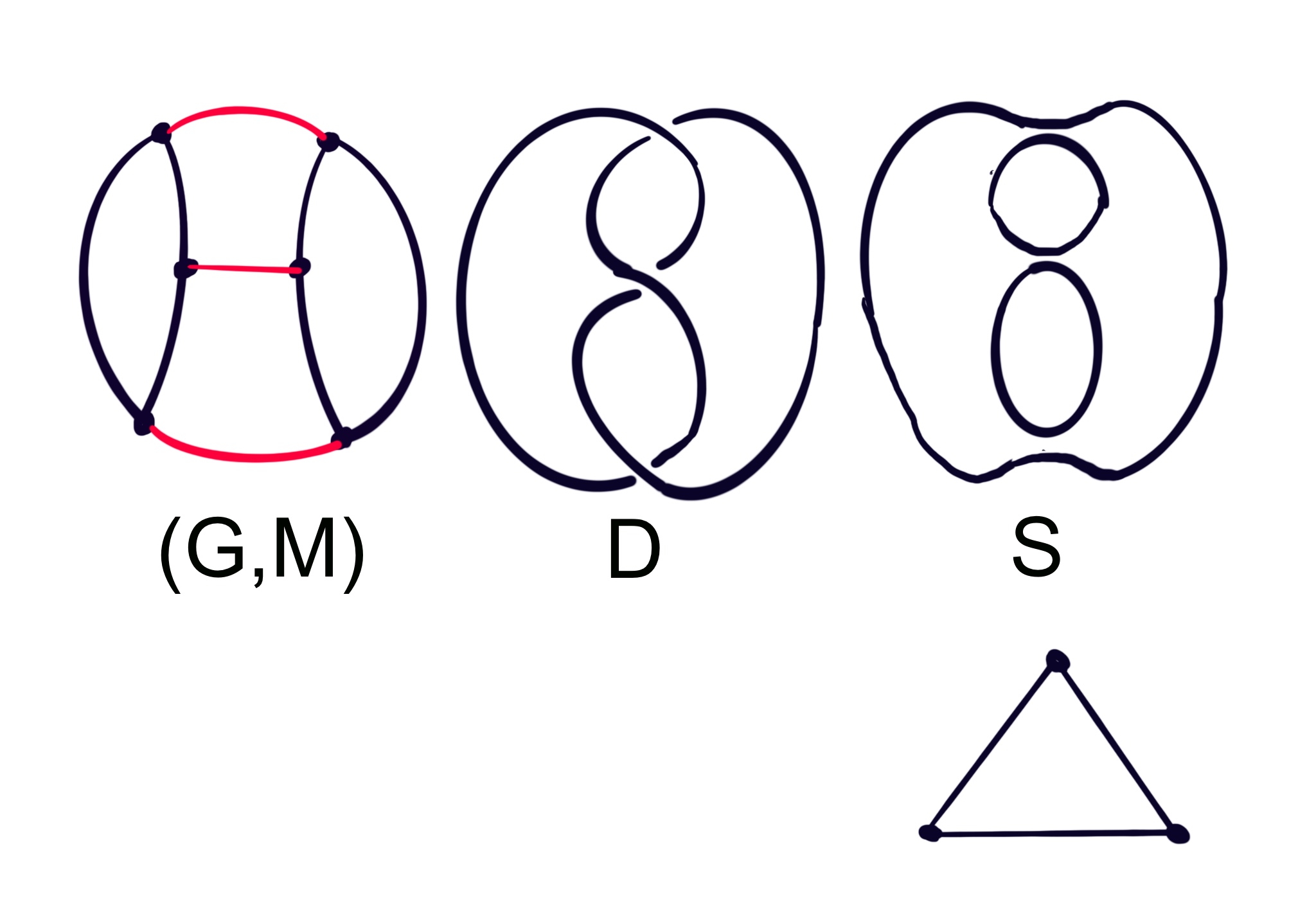}
\caption{Left-hand trefoil, colorable state and component graph}
\label{LHTrefoil}
\end{center}
\end{figure}
\end{example} 

A link diagram of the right-hand trefoil knot appears in Figure~\ref{RHTrefoil}. There are  four colorable states $S$. The component graphs $G_S$ each contribute $q(q-1)$ to the PK polynomial. Hence $P(q) = 4q(q-1)$. 

\begin{figure}[H]
\begin{center}
\includegraphics[height=2 in]{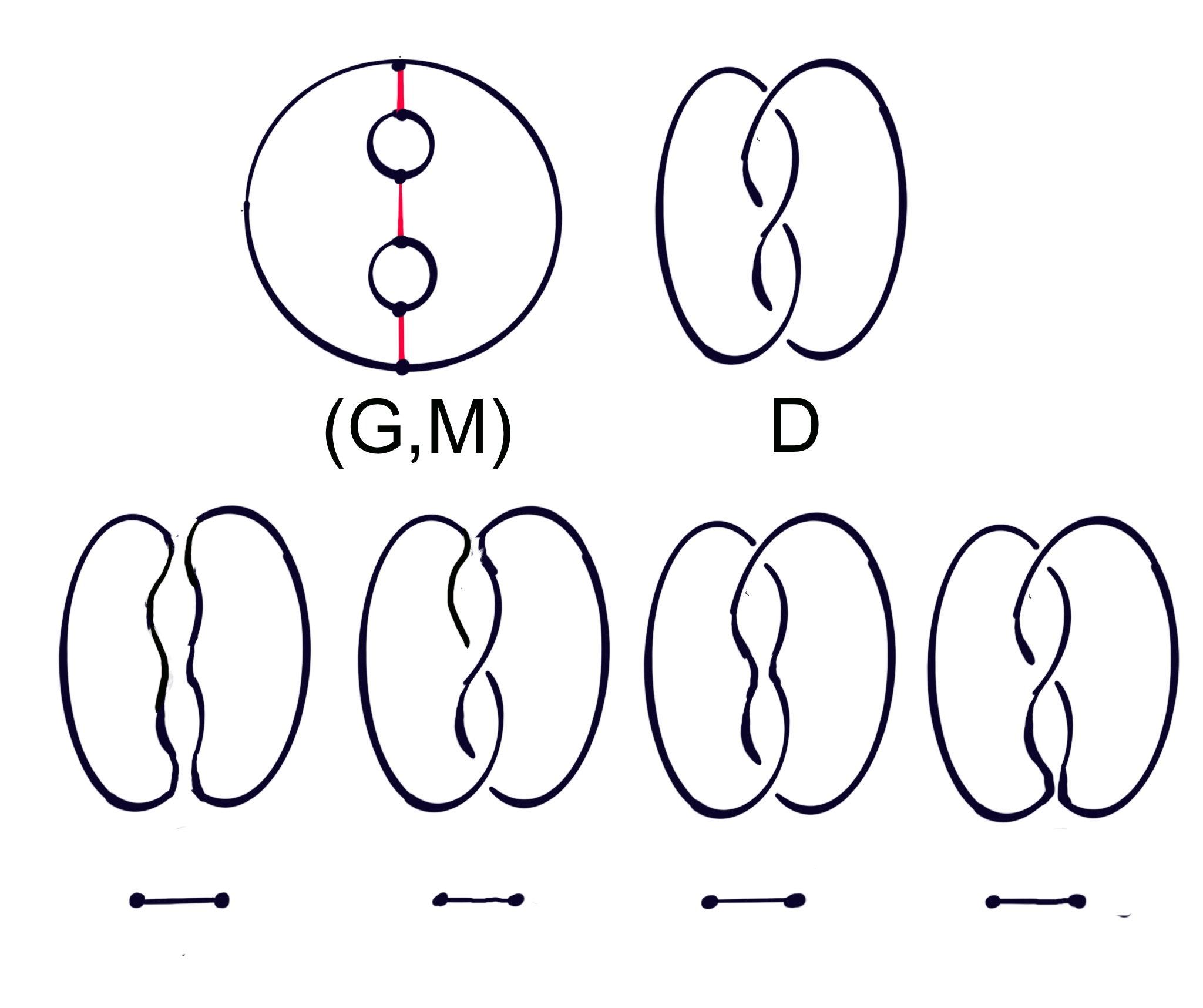}
\caption{Right-hand trefoil, colorable states and component graphs}
\label{RHTrefoil}
\end{center}
\end{figure}

The diagram is {\it reduced} if it does not have the form of (a) or (b) as in Figure \ref{Reduced}. It is {\it prime} if  it is connected, and whenever it has the form of (a), (b) or (c), the diagram inside one of the boxes is a single arc.

\begin{figure}[H]
\begin{center}
\includegraphics[height=1.2 in]{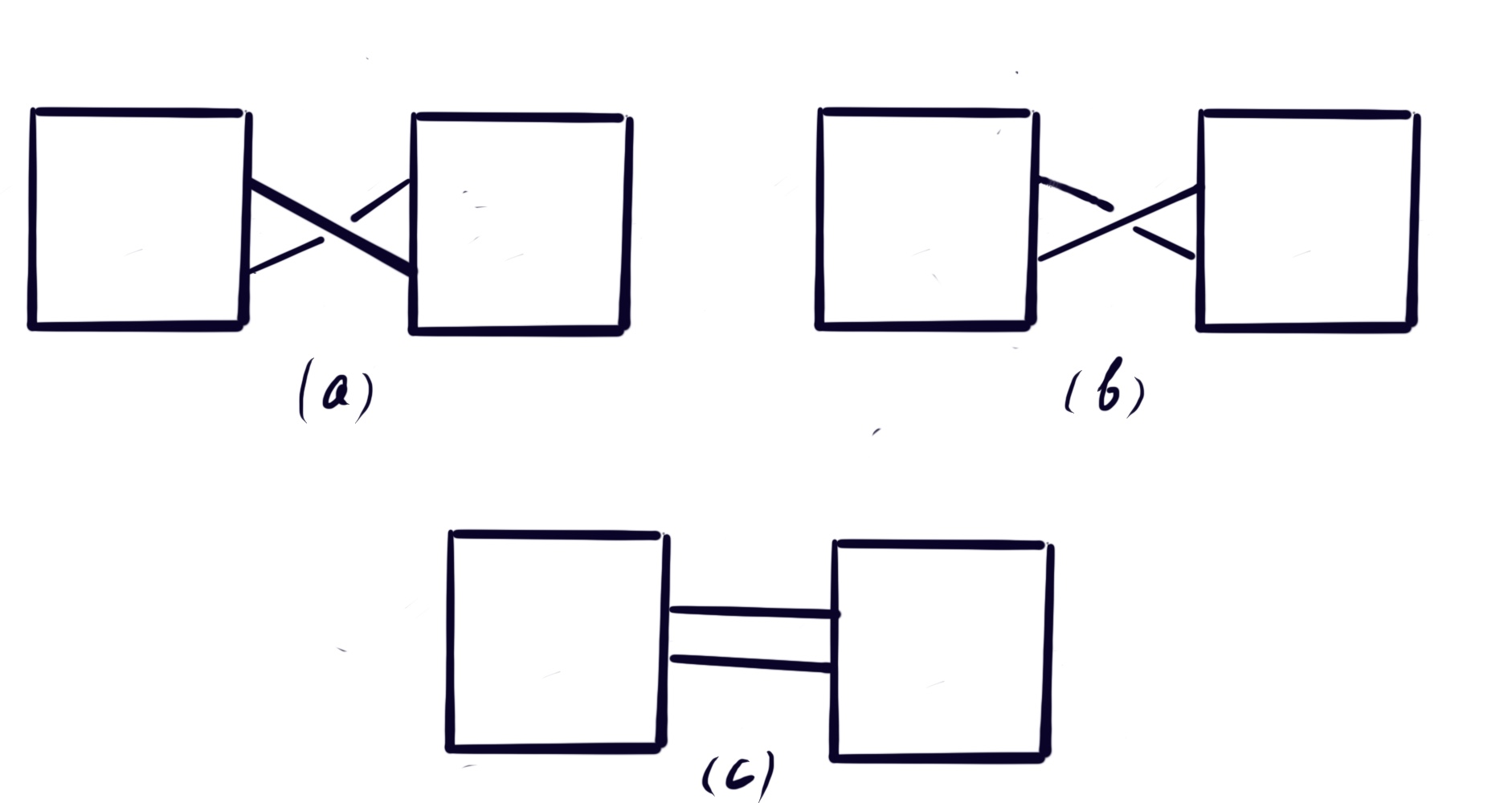}
\caption{Non-prime link diagrams}
\label{Reduced}
\end{center}
\end{figure}

\begin{definition} A {\it connected sum} of link diagrams $D_1, D_2$ is a link diagram $D_1 \sharp D_2$ obtained by cutting an arc of each diagram and connecting them as in Figure \ref{Reduced}(a), (b) or (c). \end{definition}

\begin{prop}\label{Prime} Consider a connected sum diagram $D_1 \sharp D_2$ as in Figure \ref{Reduced}. Let $P_i(q)$ be the PK polynomial of $D_i$, for $i= 1,2$. \medskip

(1) If $D$  has form (a), then $P(q) =0$.  \medskip

(2)  If $D$ has form (b)  then $P(q) = (1-q^{-1}) P_1(q)P_2(q).$ \medskip

(3) If $D$ has form (c), then $P(q) = q^{-1} P_1(q)P_2(q).$
\end{prop}

\begin{proof} Consider a connected sum diagram $D_1 \sharp D_2$ as in Figure \ref{Reduced}(c), obtained by cutting $D_i$ along an arc $a_i$. 
Each colorable state $S$ is itself a connected sum of a colorable state $S_1$ of $D_1$ and a colorable state of $S_2$ of $D_2$. Let $C_i$ be the component of $S_i$ containing the arc $a_i$. Then $C_1$ and $C_2$ combine as a single component of $S$. The contribution of $S$ to $P(q)$ is the chromatic polynomial of the graphs $G_{S_1}, G_{S_2}$ joined along the vertices corresponding to $C_1$ and $C_2$. By \cite{crapo} it is equal to the product of the chromatic polynomials of $G_{S_1}, G_{S_2}$ divided by $q$.  

The proof for a connected sum diagram as in Figure \ref{Reduced}(b) is similar. However, in this case, the vertices of $G_{S_i}$ corresponding to $C_i$, for $i=1,2$, are joined by a single edge in $G_S$.  The well-known deletion/contraction recursion for chromatic polynomials completes the argument.

For any connected sum diagram as in Figure \ref{Reduced}(a) every state has a self-crossing or self-kissing component. Hence the PK polynomial is zero.  \end{proof}

\begin{definition} A link diagram is {\it semi-reduced} if it does not have the form Figure \ref{Reduced}(a). \end{definition}

Recall that Figure \ref{Knot} gives a construction of a link diagram $D=D(G,M)$ in a surface $\Si$ from a cubic graph and perfect matching embedded in $\Si$.  
We can also invert this process, associating a cubic graph $G$ and perfect matching $M$ to a link diagram $D$. 

\begin{prop}\label{bridgeless}  The function $(G,M) \mapsto D$ induces a bijection between bridgeless cubic graphs with perfect matchings embedded in a surface $\Si$ and semi-reduced link diagrams in $\Si$. \end{prop}

\begin{proof} Suppose that $(G,M)$ is any cubic graph with perfect matching embedded in $\Si$. We observe that if $G$ has a bridge, then removing it produces two components, disjoint graphs each having an odd number of vertices. (One way to see that each has an odd number of vertices is to remove the single vertex of degree 2 in either graph to create a single edge. The resulting graph is cubic and therefore must have an even number of vertices, which implies that each of the original components had an odd number.) It follows that if $G$ has a bridge, then that edge must be in $M$, since neither component can have a perfect matching.  Hence $D$ has the form Figure \ref{Reduced} (a), and it is not semi-reduced. Conversely, any link diagram $D\subset \Si$ comes from a unique pair $(G,M)$. If $D$ is not semi-reduced, then $G$ has a bridge.

\end{proof} 

 The following is an immediate consequence of Proposition \ref{bridgeless} together with Theorems \ref{tait} and \ref{PK}. 

\begin{cor}\label{4ct} The Four Color Theorem is equivalent to the statement: Every semi-reduced link diagram in $\S^2$ has a 3-colorable state.  \end{cor}

\begin{remark} We will improve on Corollary \ref{4ct} in Theorem \ref{equivalent2} by showing that only certain prime alternating link diagrams need be considered for the Four Color Theorem.  
\end{remark} 

It might seem that we have moved far away from Guthrie's original question about coloring maps. However, the correspondence between map colorings and 3-colorings of states of a link diagram are easily seen. For this, associate the elements of the Klein 4-group with four colors. The product of any two nontrivial elements of the abelian group is the third nontrivial element while the square of any element is trivial.  As in Figure \ref{mapping}, we extract from the map a plane graph $G$, its medial graph (see \cite{diestel}), and alternating link diagram $D$. Now given any colorable state with 3-coloring, we color the finite regions of the map in the  following way. Follow a path from the infinite face of $D$ to the interior of any region, crossing arcs of $D$ transversely.  Assign to the region the product of colors of the arcs that we cross. It is easy to see that the value does not depend on the choice of path. Moreover, the assignment pulls back to a coloring of the map with four colors. The process is reversible, associating a 3-colored state to any map coloring with four colors. 

\begin{figure}[H]
\begin{center}
\includegraphics[height=3 in]{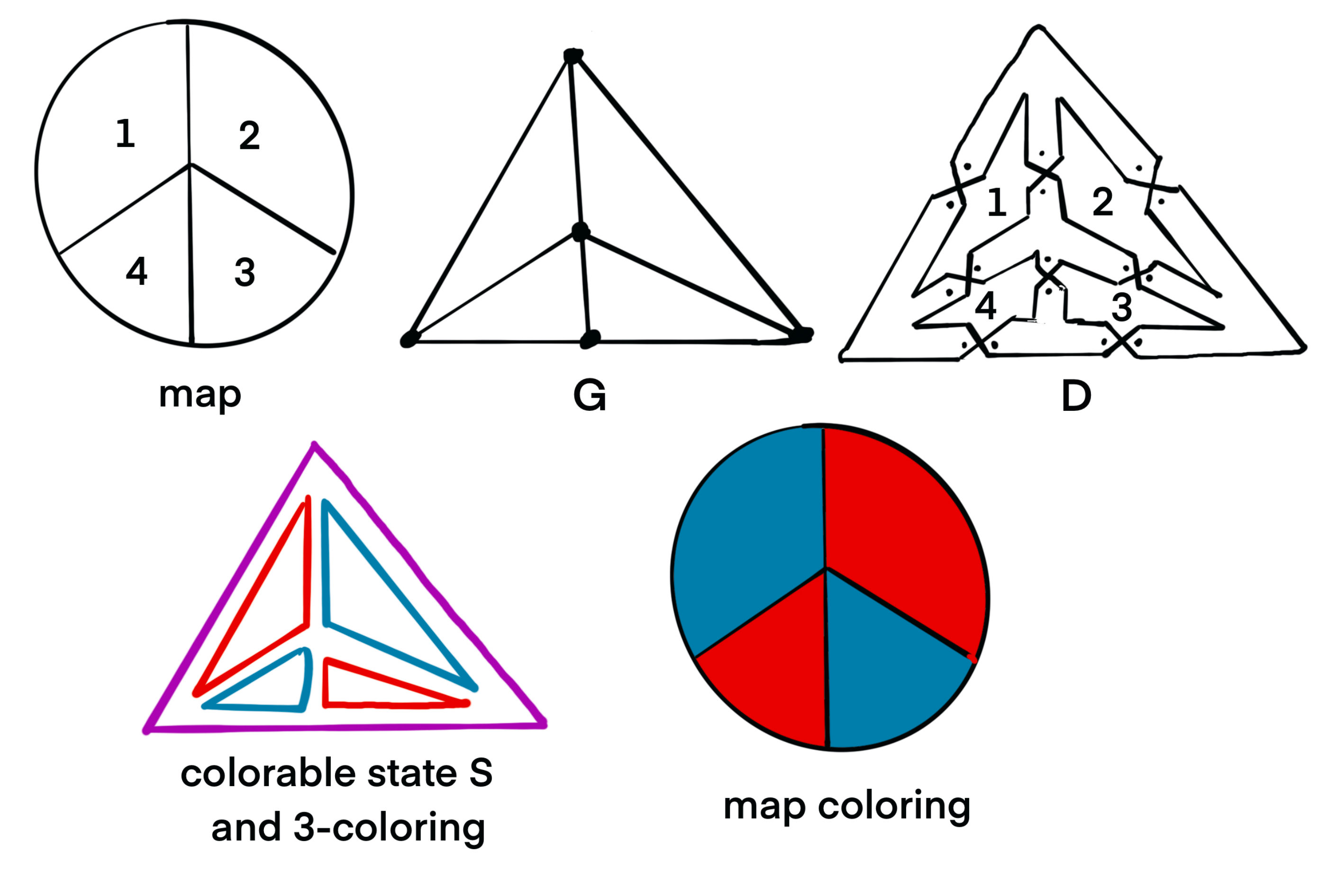}
\caption{Map coloring from 3-colorable state of associated link diagram}
\label{mapping}
\end{center}
\end{figure}


\section{Aigner's approach to the Penrose polynomial} \label{aigner's}

Let $G$ be any plane graph (not necessarily cubic) and $M(G)$ its medial graph (see \cite{diestel}). We shade the faces of $M(G)$ that $G$ bisects.  A left-right walk is a closed path that follows the medial graph, crossing at the midpoint of each edge, and continuing in the same direction. For any subset $A$ of the edge-set of $G$ consider the closed paths that are gotten from  left-right walks along $M(G)$, crossing only those edges that are contained in $A$, as in Figure \ref{LR}.  Let $c(A)$ be the number of the  paths. (For example, if $A$ is empty, then $c(A)$ is the number of faces of $G$.)  The Penrose polynomial ${\cal P}(q)$, implicit in \cite{Penrose}, is equal to $${\cal P}(q)=\sum_A (-1)^{|A|} q^{c(A)}.$$

\begin{figure}[H]
\begin{center}
\includegraphics[height=1.5 in]{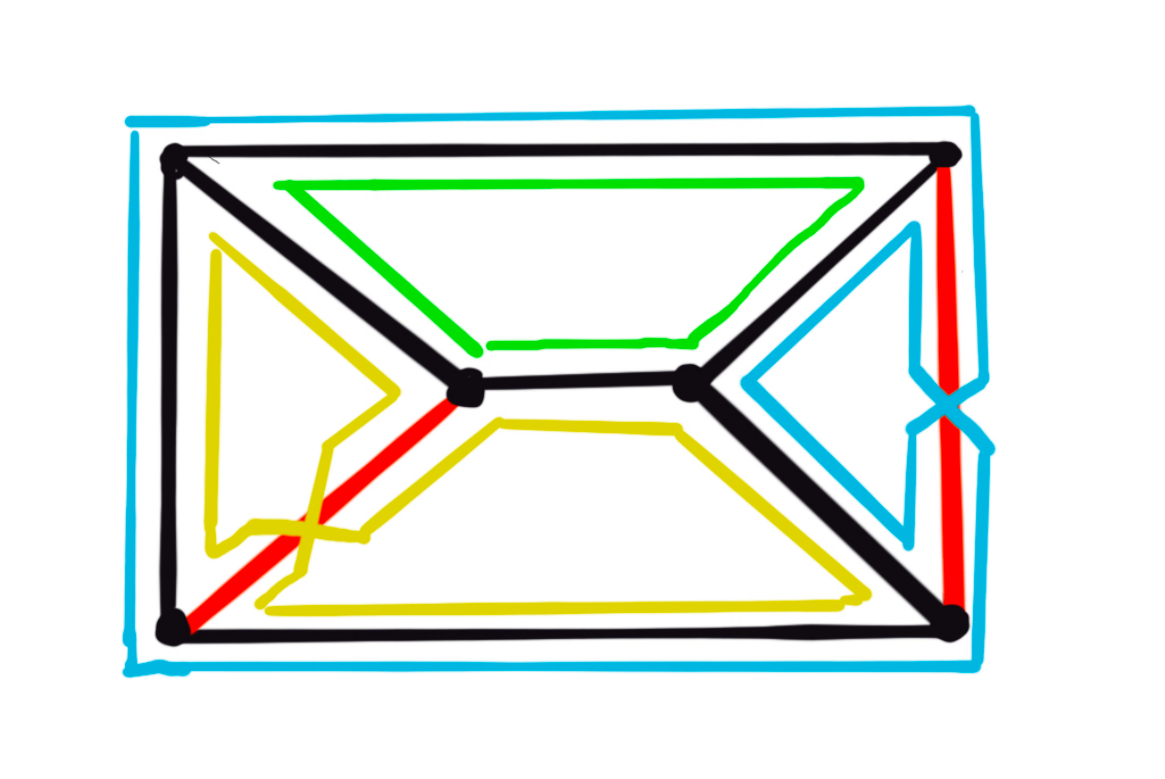}
\caption{Left-right walks corresponding to the set $A$ consisting of the two red edges. Here $c(A)=3$.} 
\label{LR}
\end{center}
\end{figure}

Penrose's motivation was to prove what in 1969 was called the Four Color Conjecture. When $G$ is a plane cubic graph, the value ${\cal P}(3)$ is the number of proper edge 3-colorings of $G$. Hence for Penrose,  proving the conjecture was equivalent to showing that ${\cal P}(3)$ is nonzero for any bridgeless plane cubic graph.

An {\sl admissible $n$-valuation} of $M(G)$ is a coloring of its edges with $1, \ldots, n$ such that at every vertex exactly two colors appear in one of the two configuations shown in Figure \ref{valuation}. Proposition 4 of \cite{aigner} implies that ${\cal P}(n)$  is equal to the number of admissible $n$-valuations of $M(G)$. 

\begin{figure}[H]
\begin{center}
\includegraphics[height=1.5 in]{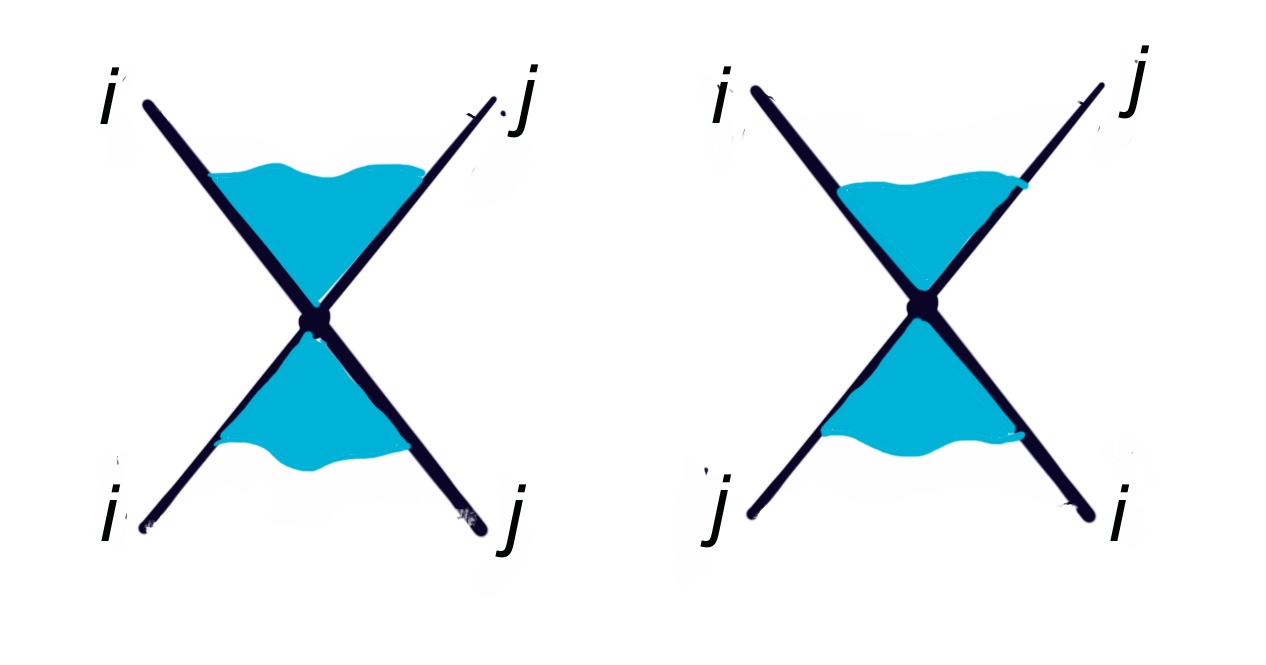}
\caption{Admissible $n$-valuation labeling}
\label{valuation}
\end{center}
\end{figure}

Every Penrose polynomial arises as a PK polynomial. To see this, ``blow up" each vertex of $G$ to obtain a cubic graph $\hat G$ (see Figure \ref{augment2}). The original edges of $G$ provide a perfect matching $M$ for $\hat G$. Now $n$-valuations of $M(G)$ correspond bijectively to Tait $n$-colorings of $(\hat G, M)$. (Here the medial graph $M(G)$ is the projection of the link diagram $D$ with each crossing point replaced by a 4-valent vertex.) Hence Aigner's polynomial ${\cal P}(q)$ for $G$ is the same as the PK polynomial $P(q)$ for $(\hat G, M)$.

A link diagram in a surface is {\it alternating} if, whenever we follow along a component, we encounter over- and under-crossings alternately. It is not difficult to see that the  link diagram corresponding to $(\hat G, M)$ is alternating. In fact, every alternating link diagram arises this way. Consequently, proving the Four Color Theorem  amounts to showing that every semi-reduced alternating link diagram in the plane has a 3-colorable state.  Moreover, Proposition \ref{Prime} implies that one can restrict attention to prime diagrams. Summarizing:

\begin{theorem}\label{equivalent} The Four Color Theorem is equivalent to the statement: Every reduced prime alternating link diagram in the plane has a 3-colorable state. \end{theorem}

\begin{figure}[H]
\begin{center}
\includegraphics[height=1.5 in]{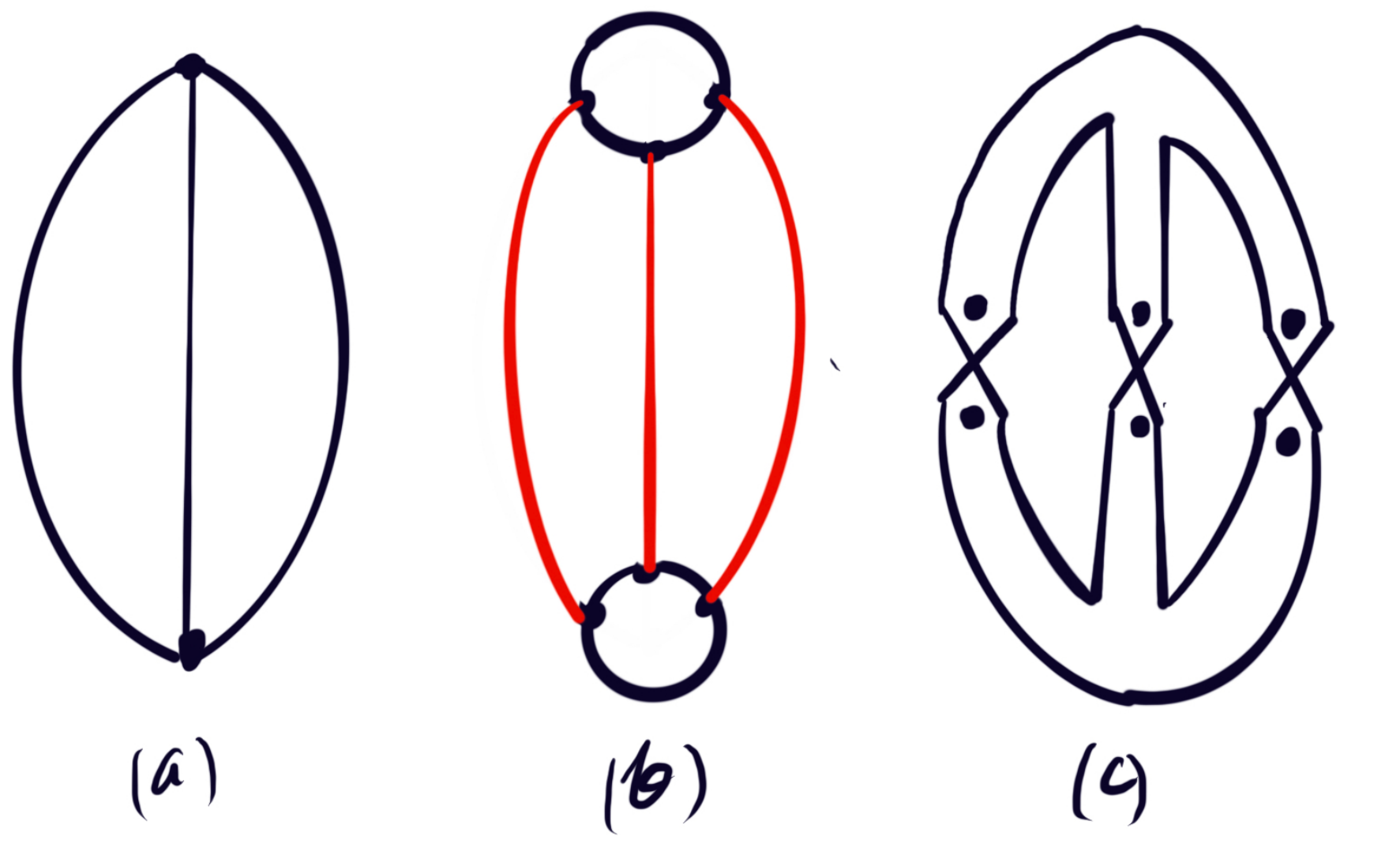}
\caption{(a) Plane graph $G$; (b) blowing up vertices of $G$ to obtain $\hat G$; (c) medial graph $M(G)$}
\label{augment2}
\end{center}
\end{figure}

\begin{remark} While every polynomial arising from Aigner's construction arises as a PK polynomial, the reverse inclusion does not hold. Figure \ref{nonaltex} displays a graph and perfect matching $(G,M)$ such that the associated link diagram is not alternating. Its PK polynomial has odd linear coefficient. Proposition \ref{linearcoef} implies that the polynomial does not arise from Aigner's construction.  \end{remark}

\section{The Smoothing Space of a Link Diagram} 

Assume that $D$ is a link diagram  in a surface $\Si$. Recall that a (smoothing) state $S$ of $D$ is the diagram obtained by smoothing any subset of crossings of $D$ as in Section \ref{PKPoly}.

\begin{definition} If $D$ is a link diagram in a surface $\Si$, its {\it smoothing space} is the metric space ${\cal S}(D)$ consisting of all  states $S$ of $D$ together with the distance function $d: {\cal S}(D) \times {\cal S}(D) \to \N$ with $d(S, S')$ equal to the number of crossings of $D$ at which $S$ and $S'$ differ.   \end{definition} 

\begin{remark} If $D$ has $n$ crossings, then the smoothing space ${\cal S(D)}$ can be identified with the vertices of the $n$-dimensional cube with unit-length edges. The distance between two states is the length of the shortest edge-path between them. See Example \ref{cube}  below. \end{remark}

Define the function $||\cdot||: {\cal S}(D) \to \N$ such that for any $S \in {\cal S}(D)$, the value $||S||$ is the number of components of the link described by $S$. 

\begin{theorem} \label{maxima} The colorable states $S$ of ${\cal S}(D)$ are the strict local maximum points of $||\cdot||$.  \end{theorem} 

\begin{proof} Suppose that $S$ is a colorable state. Consider a neighborhood $N$ in $\Si$ containing a single crossing of $D$. By the definition of colorable state, the two arcs of $S \cap N$ belong to different components of $S$. If they cross, then smoothing the crossing would produce a state $S'$ in which the two components in $N$ combine. In this case, $||S|| > ||S'||$. Similarly if the two arcs do not cross, then again they must belong to different components;  reversing the smoothing  would produce a state $S'$ with $||S|| > ||S'||$. Since we have considered every state $S'$ such that $d(S', S) \le 1$, we see that $||S||$ is a strict local maximum value. 

Conversely, suppose that $S$ is a state such that $||S||$ is a strict local maximum value. The two arcs of $S \cap N$ must belong to different components of $S$ since otherwise smoothing or reverse-smoothing at the site could not produce a state $S'$ with $||S|| > ||S'||$. Hence $S$ is colorable. 
\end{proof} 

\begin{theorem} The number of colorable states of a link diagram $D$ with $n>0$ crossings is at most $2^{n-1}$. For plane diagrams it is equal to $2^{n-1}$ if and only if $D$ is the reduced diagram of the right-hand $(2, n)$ torus knot or link. \end{theorem}

\begin{proof} It suffices to show this for connected $D$. By Theorem \ref{maxima} the colorable states form a set of vertices in the cube ${\cal S}(D)$ that have distance greater than one. Such a set has cardinality no greater than $2^{n-1}$. 

If $D$ is a plane graph, every colorable state $S$ consists of a set of circles embedded in the plane. Any two circles cross at an even number of points, so colorable states have an even number of crossings. If $2^{n-1}$ states are colorable, then conversely, all states with an even number of crossings are colorable. 

Now applying Theorem \ref{maxima}, we see that if $||S_0||=c$, then all states with an even number of crossings have $c$ components, while those with an odd number have $c-1$ components. But we can easily produce a state with just one component by successively adding crossings to connect the components of $S_0$. Hence $c=2$ and the state $S_0$ consists of two circles that kiss $n$ times. Then $D$ is a reduced $n$-crossing diagram of the right-hand $(2,n)$-torus knot or link.

\end{proof} 

\begin{example} 
A diagram $D$ of a 2-component link in the torus appears in Figure \ref{cube}(a). Its crossings are ordered. The eight possible states $S$ of $D$ can be encoded by elements of $\{0,1\}^3$. Here the coefficients of the triple are ordered as the crossings of the diagram. The symbol $0$ indicates that the crossing is smoothed while $1$ indicates that it is left as a crossing. For example, the symbol $(1,0, 1)$ indicates that only crossing $ii$ is smoothed.  The reader can verify the values $||S||$ placed at the vertices of the 3-dimensional cube in Figure \ref{cube}(b). We note that there are exactly two states at which $||\cdot||$ achieves a local strict maximal value. For each of them the contribution to the PK polynomial $P(q)$ of the diagram is $q(q-1)^2$. Hence $P(q) = 2 q(q-1)^2$.

\begin{figure}[H]
\begin{center}
\includegraphics[height=1.8 in]{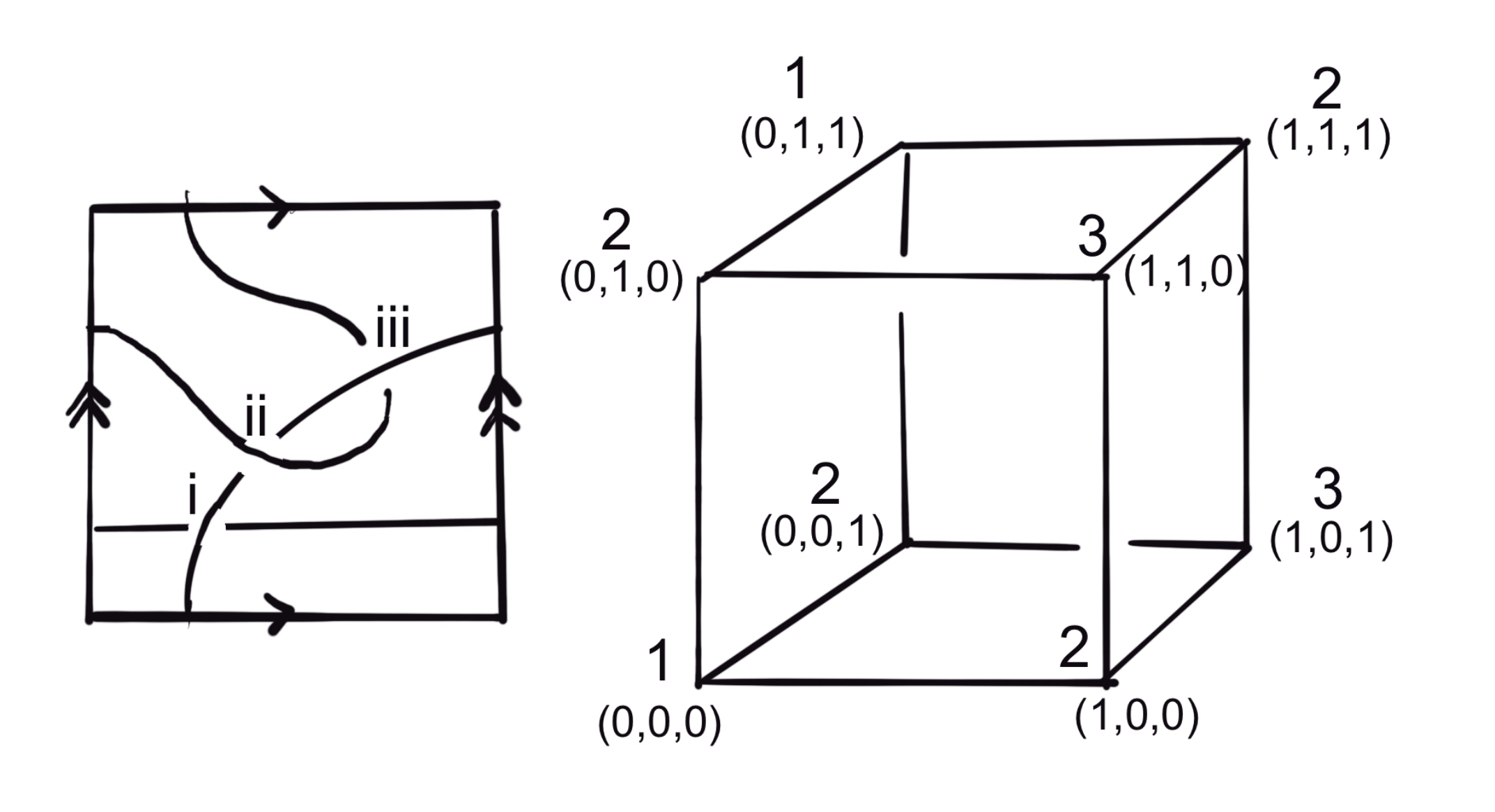}
\caption{Link diagram $D$ in the torus (left) ; states of $D$  (right) }
\label{cube}
\end{center}
\end{figure}

\end{example}


 \section{Alternating link diagrams}  
 
An alternating link diagram $D$ can be checkerboard shaded so the regions containing dots, the {\it dotted regions of $D$}, are the shaded regions. We will refer to the regions without dots as {\it undotted regions}.\bs

\begin{definition} The {\it Tait graph} of an alternating plane link diagram is the graph with vertices corresponding to dotted regions, with an edge through each crossing, joining the vertices of the corresponding regions. 
The {\it dual Tait graph} is the dual graph with vertices corresponding to undotted regions of $D$. \end{definition}

For an alternating plane link diagram $D$, the {\it all-smoothed state} $S_0$ consists of disjointly embedded circles corresponding to the undotted regions, and hence to the vertices of the dual Tait graph $T$. Thus $S_0$ is colorable. 

The following proposition should be compared with Corollary 3 of \cite{aigner}.

\begin{prop}\cite{aigner} \label{leading} Assume that $D$ is a semi-reduced alternating link plane diagram and $T$ its dual Tait graph. Then the degree of its PK polynomial $P(q)$ is $||S_0||$.  Moreover, if $T$ is simple, then $P(q)$ is monic. \end{prop}

\begin{proof}  By Proposition \ref{PK},  each colorable state $S$ of $D$ contributes to $P(q)$ the chromatic polynomial of the component graph $G_S$. The degree of this polynomial is the number $||S||$ of components of $S$. Hence for the first assertion it suffices to show that no state of $D$ has more components than $S_0$. 

Any state $S$ is obtained from $S_0$  by adding back some of the crossings of $D$. 
Thus the components of $S$ correspond to left-right walks on a subgraph of the dual Tait graph that omits a subset of edges. As in \cite{sw}, the number of left-right walks is equal to the nullity of the mod-2 Laplacian matrix. Since the nullity cannot exceed the size $||S_0||$ of the matrix, the proof of the first assertion is complete. 

If $T$ is simple, then the Laplacian matrix of $S$ is already reduced modulo 2. No other state has zero Laplacian matrix, so no other state contributes a term of degree $||S_0||$. Since the chromatic polynomial is monic, the second assertion of Proposition \ref{leading} follows.  \end{proof} 

For $D, T$ as above, let $\bar T$ be a simple graph obtained from $T$ by removing all but one edge between each pair $v, w$ of joined vertices. (The graph $\bar T$ is not necessarily unique.) Let $r$ be the number of removed edges, so $r = \sum (r(v,w)-1)$, where $r(v,w)$ is the number of edges in $T$ joining $v$ and $w$.  Let $\bar D$ be the link diagram with dual Tait graph $\bar T$.

We prove another result that appears in a different form in \cite{aigner}. 

\begin{theorem}\cite{aigner} \label{factor} Let $D$ be a semi-reduced alternating plane link diagram, and $\bar D, r$ as described above. The PK polynomials $P(q),\bar P(q)$ of $D$ and $\bar D$ satisfy 
$$P(q)= 2^r \bar P(q)$$. \end{theorem} 

\begin{proof} We define a function $f$ from the states of $D$ to states of $\bar D$ by reducing modulo 2 the number of edges between any pair of vertices of $T$. The mod-2 Laplacian matrix of $S$ is the Laplacian matrix of $f(S)$, and so $S$ and $(S)$ have the same number of components. The function $f$ is $2^r$-to-$1$, since for each pair $(v,w)$ of vertices of $T$ there are $2^{r(v,w)-1}$ ways to remove an even subset of edges between them. 

By Theorem \ref{maxima}, a state $S$ of $D$ is colorable if and only if $||S||$ is a strict local maximum value of $|| \cdot ||$. This is equivalent to the condition that the nullity of the mod-2 Laplacian matrix decreases if we add or remove the contribution of any edge. But that is equivalent to the same condition on the mod-2 Laplacian matrix of $f(S)$. Hence $S$ is colorable if and only if $f(S)$ is. 

Finally, we claim that chromatic polynomial of the component graph of $S$ is the same as that of $f(S)$. For this, it suffices to show that if a state $S'$ is obtained from $S$ by removing a pair of edges $e_1, e_2$ between vertices $v, w$, then $S$ and $S'$ have the same number of $n$-colorings for every positive integer $n$ (see Definition \ref{ncoloring}). 

The edges $e_1, e_2$ bound a disk $\Delta$ in the plane. Suppose that $S$, and hence $S'$, are colorable, and consider any $n$-coloring of $S'$ as shown in Figure \ref{Colorable}. The colors $a, b, c, d$ must satisfy $a \ne b$ and $c \ne d$. A component that enters $\Delta$ must leave it, so either $a=c$ and $b =d$ or else $a=d$ and $b=c$. For any $n$-coloring of $S$, we obtain a corresponding $n$-coloring of $S'$ by exchanging the colors $a$ and $b$ inside $\Delta$ (which automatically exchanges $c$ and $d$). This exchange can have no effect on the diagram outside $\Delta$, and it gives a bijection of $n$-colorings of $S$ and $S'$. \end{proof} 

\begin{figure}[H]
\begin{center}
\includegraphics[height= 1.5 in]{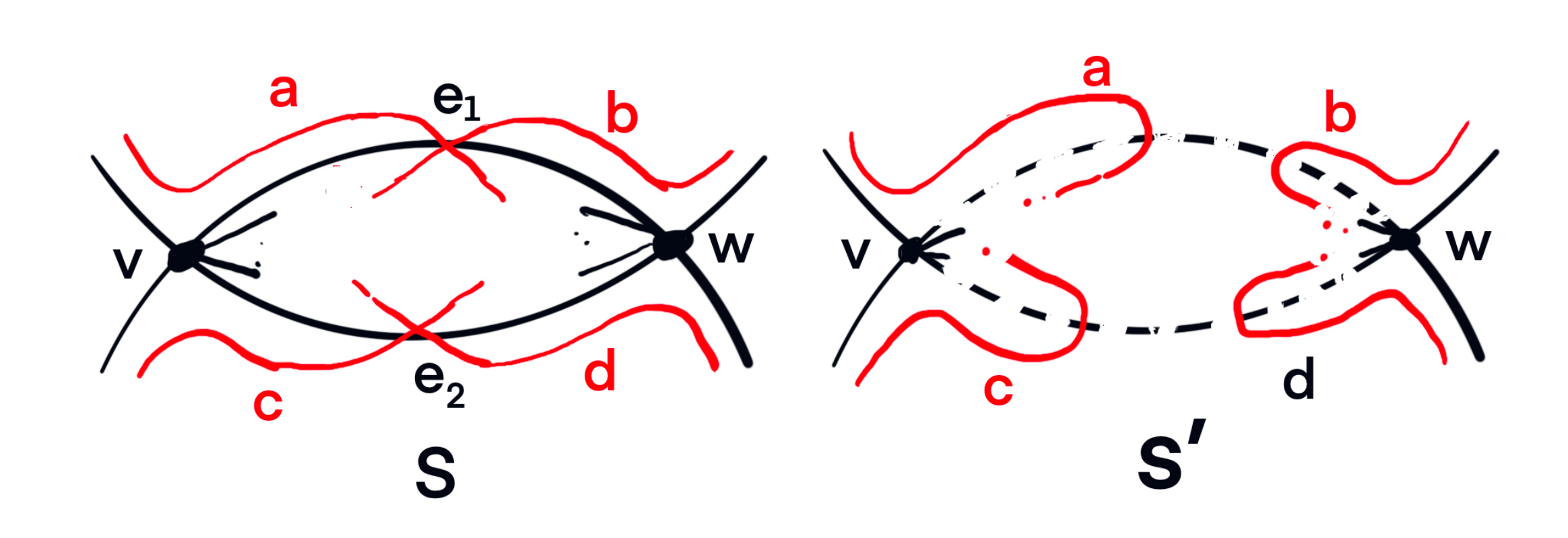}
\caption{Obtaining an $n$-coloring of $S'$ from $S$}
\label{Colorable}
\end{center}
\end{figure}

We see from the argument above that the $2^r$-to-$1$ function $f$ from states of $D$ to states of $\bar D$ also preserves the number of components of the state. This yields:

\begin{cor} For a semi-reduced alternating plane diagram $D$, the number of colorable states with a fixed number of components is a multiple of $2^r$. \end{cor}

\begin{example} Proposition \ref{linearcoef} requires that the diagram be alternating. The non-alternating diagram in Figure \ref{nonaltex} has three colorable states, each with three components. The PK polynomial is $P(q) = 3q^3-8q^2+5q$. 

\begin{figure}[H]
\begin{center}
\includegraphics[height=1.5 in]{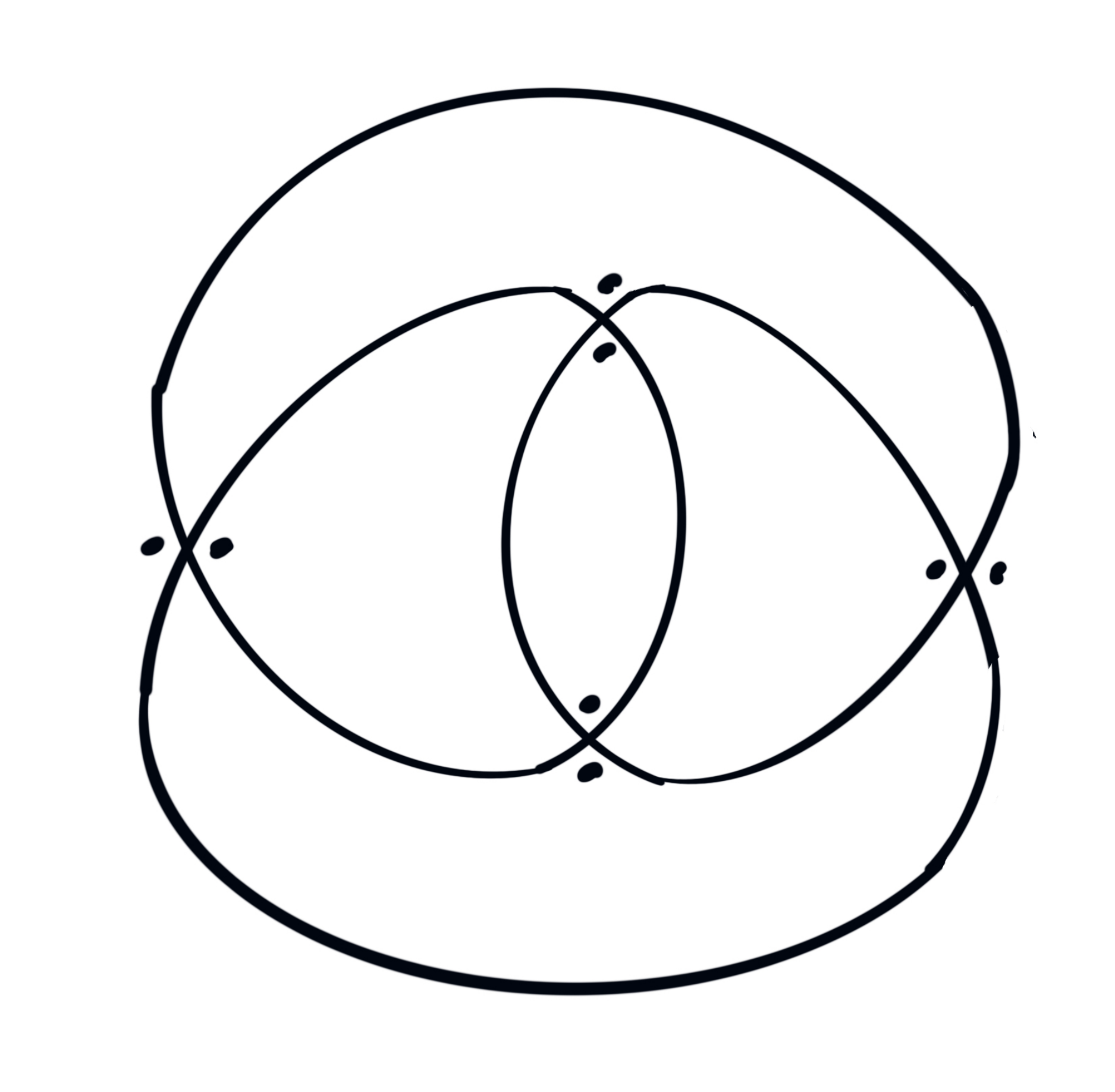}
\caption{Nonalternating diagram such that $P(q)$ has odd leading coefficient }
\label{nonaltex}
\end{center}
\end{figure}

\end{example}

Let $D$ be a semi-reduced alternating plane link diagram. If 
$D$ has a chain of (one or more) dotted bigons, they correspond to parallel edges in the dual Tait graph of $D$. If we replace each chain by a single dotted crossing, then the simplified diagram $\bar D$ in Theorem \ref{factor} remains unchanged, and so the PK polynomial is changed by a nonzero constant factor.

What can be done with undotted regions? We show that for the purpose of proving the Four Color Theorem, we can restrict our attention to reduced alternating prime link diagrams without {\sl any} bigon regions. 

Consider a semi-reduced alternating link diagram with an undotted bigon on the left of Figure \ref{F1} (left) and the diagram $D'$ on the right obtained by smoothing the two crossings of the bigon  and then deleting the resulting unknotted circle. One checks easily that $D'$ is alternating. 

\begin{figure}[H]
\begin{center}
\includegraphics[height=2.5 in]{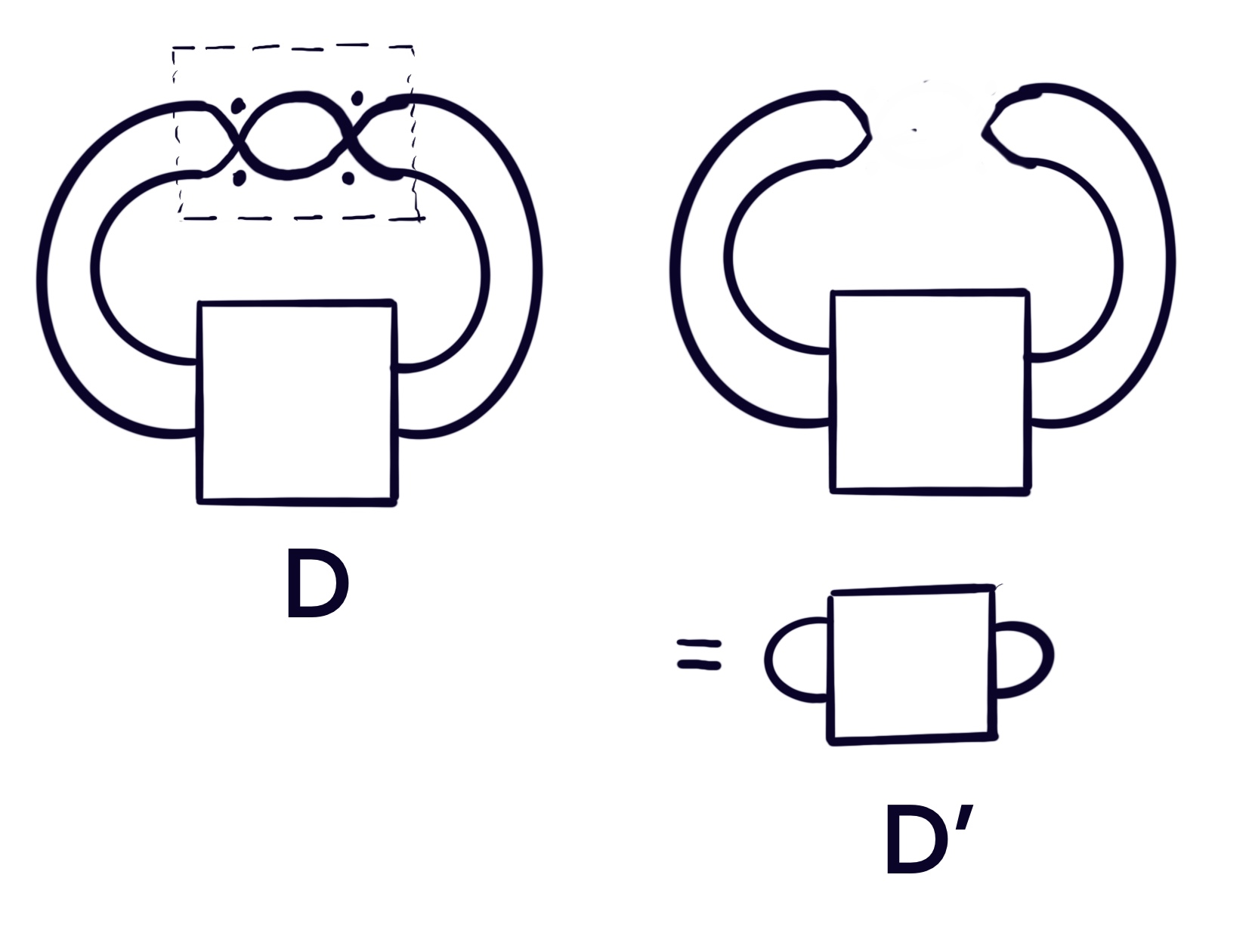}
\caption{Diagram with undotted bigon (left); after removal (right)}
\label{F1}
\end{center}
\end{figure}

The diagram $D'$ is also semi-reduced. This can be seen in the following way. If $D'$ were not semi-reduced, then it and $D$ would appear as in Figure \ref{F2}. For any alternating link diagram, shading regions containing one or more dots results in a checkerboard coloring of the diagram. If we shade those regions of $D$ we do not obtain a checkerboard coloring. Hence $D'$ must be semi-reduced.

\begin{figure}[H]
\begin{center}
\includegraphics[height=1.3 in]{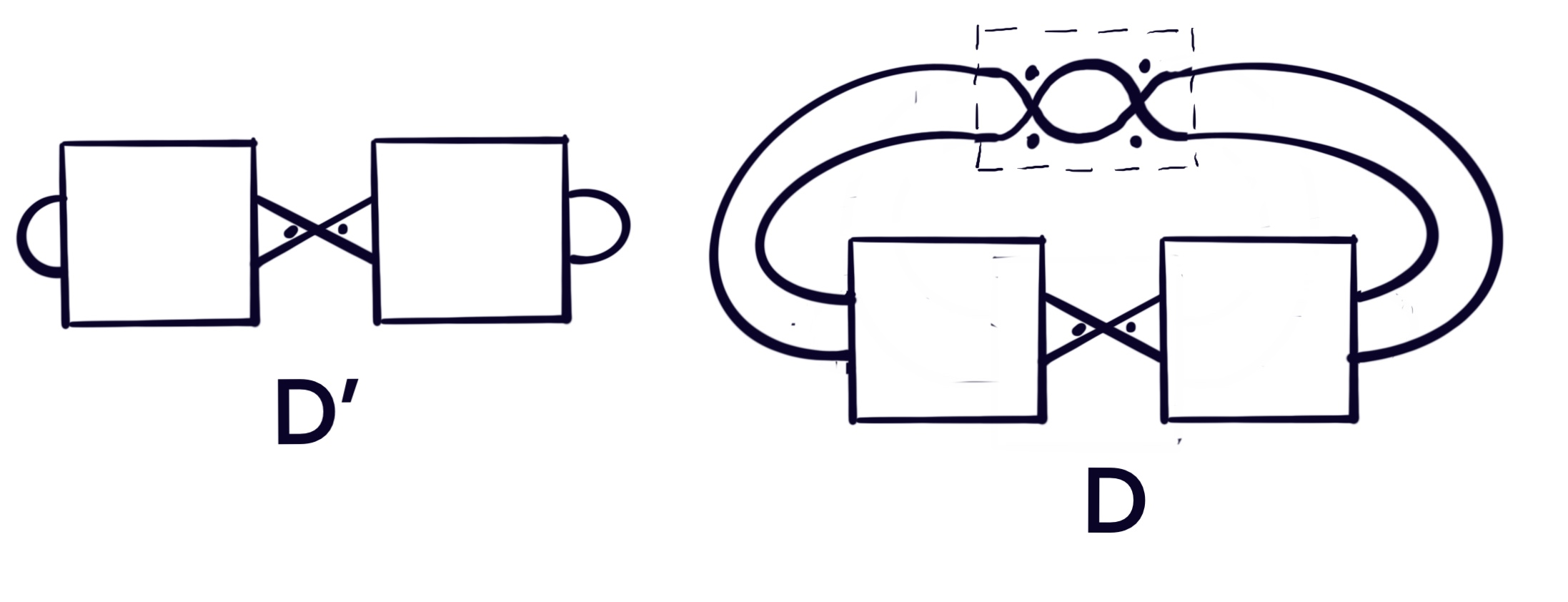}

\caption{Proof that $D'$ is semi-reduced}
\label{F2}
\end{center}
\end{figure}

Suppose that $D'$ has a 3-colorable state $S'$. Extend $S'$ to a state $S$ of $D$ by smoothing the two crossings of the bigon to create an unknot component. Any 3-coloring of $S'$ can be extended to a 3-coloring of $S$ by giving the unknot a color different from that of the two adjoining arcs. 

The above argument can be used repeatedly on undotted bigons. Removing bigons might produce new ones, but the process must terminate since the crossing number is reduced at each step.

 Hence we have: 

\begin{theorem}\label{equivalent2} The Four Color Theorem is equivalent to the statement: Every reduced, prime alternating link diagram in the plane {\sl without any bigon regions} has a 3-colorable state. \end{theorem}

\section{An invariant of alternating links} 

One of Tait's famous conjectures is that given any two reduced alternating diagrams of the same prime link, one can be transformed into the other by performing a sequence of {\it flypes} (see Figure \ref{Flype1}). The conjecture was proved in 1991 by  Menasco and Thistlethwaite \cite{mt}. 

\begin{figure}[H]
\begin{center}
\includegraphics[height=.85 in]{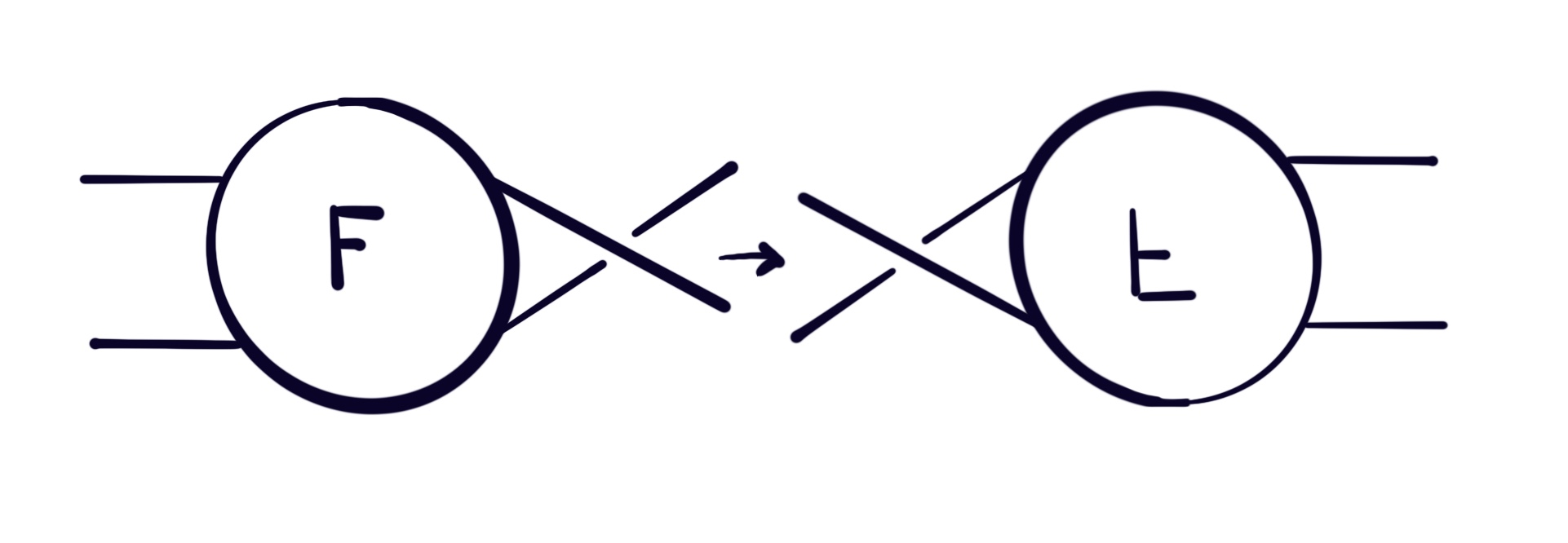}
\caption{Flype}
\label{Flype1}
\end{center}
\end{figure}

\begin{prop} \label{flype} Assume that $D$ is a link diagram in a surface. If $D'$ is obtained from $D$ by a finite sequence of flypes, then the diagrams have the same PK polynomial. 
\end{prop} 

\begin{proof} It suffices to consider the case in which $D'$ is obtained from $D$ by a single flype. 
For any state $S$ of $D$, let $S'$ be the state of $D'$ obtained by smoothing the same crossings (flipped in the tangle).  We see that  $S$ is colorable if and only if the corresponding state $S'$ of $D'$ is colorable. Moreover, the associated graphs $G_S, G_{S'}$ are isomorphic, and hence the chromatic polynomial contributions to the PK polynomial are the same. \end{proof}

By \cite{mt} and Proposition \ref{flype} we have: 

\begin{theorem} Assume that $L \subset \S^3$ is an alternating link. The PK polynomial of any reduced alternating diagram of $L$ is an invariant of the link.    \end{theorem} 

As a consequence of the theorem, the following is well defined. 

\begin{definition} The {\it coloring polynomial} of a prime alternating link $L \subset \S^3$ is the PK polynomial of any reduced alternating diagram of $L$. \end{definition} 

\begin{example} By Example \ref{trefoils}, the coloring polynomial of the right- and left-hand trefoil knots are $4q(q-1)$ and $q(q-1)(q-2)$, respectively. In particular, the two knots have different numbers of 3-colorings: 24 for the right-hand trefoil but only 6 for the left-hand. From this we see the well-known result that the trefoil knot is {\it chiral} (i.e., it is not isotopic to its mirror image). \end{example} 

The example above shows that chirality of prime alternating links $L$ can sometimes be proven using the coloring polynomial. If the coloring polynomial of $L$ is different from that of $\overline L$, then $L$ is chiral. 

Proposition \ref{leading} suggests a simple way to compare their degrees: Let $D$ be a reduced alternating diagram of $L$. The number of undotted regions  (the number of vertices of the dual Tait graph) is the degree of the coloring polynomial of $L$. The number of dotted regions (the number of vertices of the Tait graph) is the degree of the polynomial of $\overline L$. If the two numbers are different, then $L$ is chiral. This is a well-known result that follows from Theorem 3.1 of \cite{kauffman2}).

\begin{example} A projection of the link $L=8_1^3$  appears at the top of Figure \ref{3link}. There are 3 undotted regions and 7 dotted regions. Hence $L$ is chiral. 

\begin{figure}[H]
\begin{center}
\includegraphics[height=1.7 in]{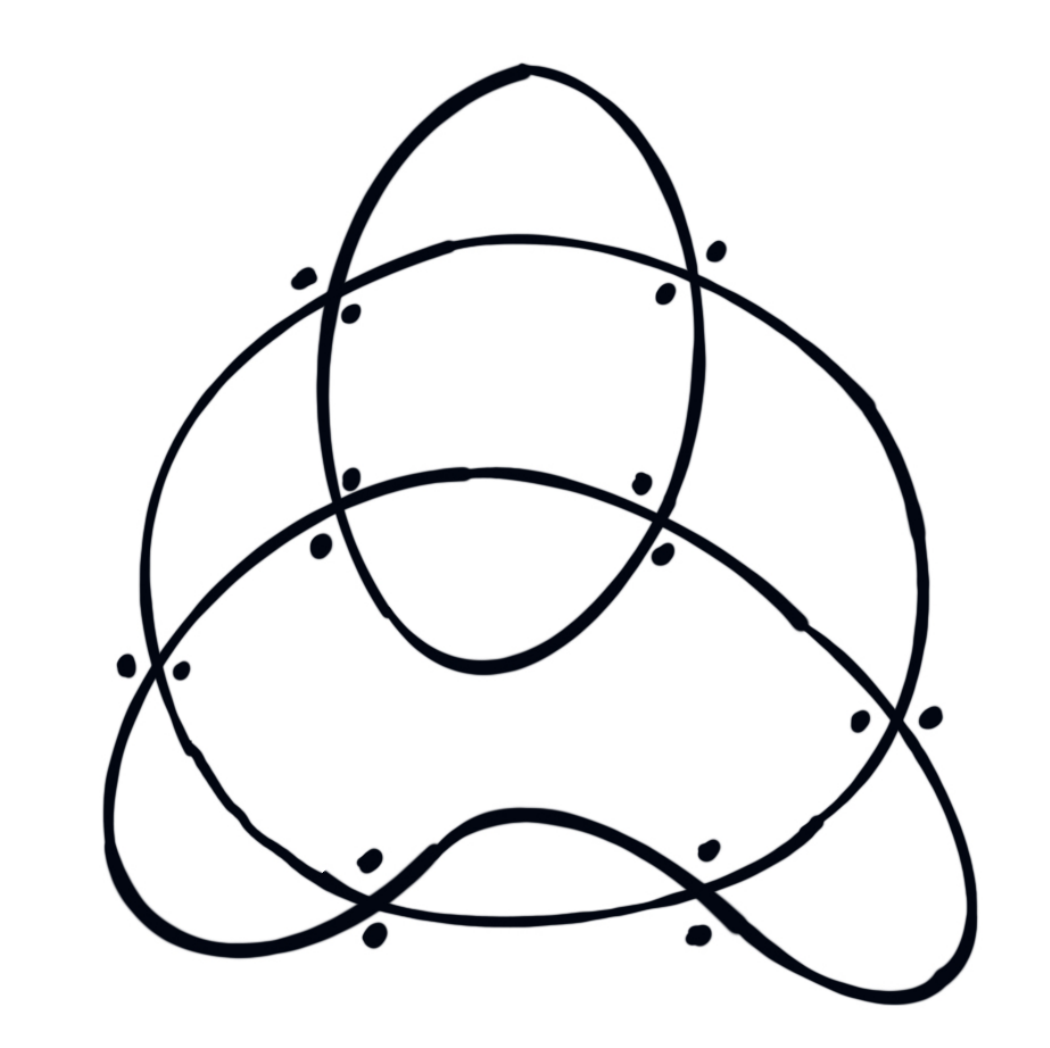}
\caption{Projection of $L=8_1^3$  }
\label{3link}
\end{center}
\end{figure}

\end{example}


\section{PK Polynomial and Falling Factorials} 
Assume that $(G,M)$ is a cubic graph with perfect matching $M$. 
For each $m \in \N$, let $c_m$ denote the number of Tait $m$-colorings of $(G,M)$ using all $m$ colors. 
(Note that $c_m =0$ for sufficiently large $m$.)  The number of Tait $n$-colorings of $(G,M)$ is: 
$$\sum_m {n\choose m} c_m.$$
Then the PK polynomial $P(q)$ is:
$$P(q) = \sum_m {q\choose m} c_m.$$

For any nonnegative integer $m$, the {\it $m$th falling factorial} $(q)_m$ is the polynomial $q(q-1)\cdots (q-m+1)$. Expressing ${q\choose m}$ as $(q)_m/m!,$ 
we have the following formulations of $P(q)$: 
\begin{equation} \label{first} P(q) = \sum_m {{c_m}\over m!} (q)_m \end{equation}  
and:
\begin{equation} P(q) = \sum_m  e_m (q)_m, \end{equation}  
where $e_m$ is the number of ways to color $(G,M)$ with exactly $m$ colors, up to renaming of colors.\bs

When $P(q)$ is written in terms of the falling factorials $(q)_m$ instead of the standard basis for polynomials, the degree and leading coefficient remain the same. The following is immediate.

\begin{prop} The  degree of $P(q)$ is the largest number $n$ such that $(G,M)$ can be  Tait $n$-colored using all $n$ colors. The leading coefficient of $P(q)$ is $e_n$. \end{prop}

 \begin{example}\label{adequacy} Recall that for semi-reduced alternating plane link diagrams  the degree of the PK polynomial is the number $||S_0||$ of components of the all-smoothed state. For general link diagrams the degree of $P(q)$ can be strictly less than $||S_0||$. Consider the link diagram in Figure \ref{cexample} below. One checks that $S_0$ has 3 components but no 3-component state of $D$ is colorable. The only colorable state is the link diagram itself. Hence $P(q) = q(q-1)$.

\begin{figure}[H]
\begin{center}
\includegraphics[height=1.5 in]{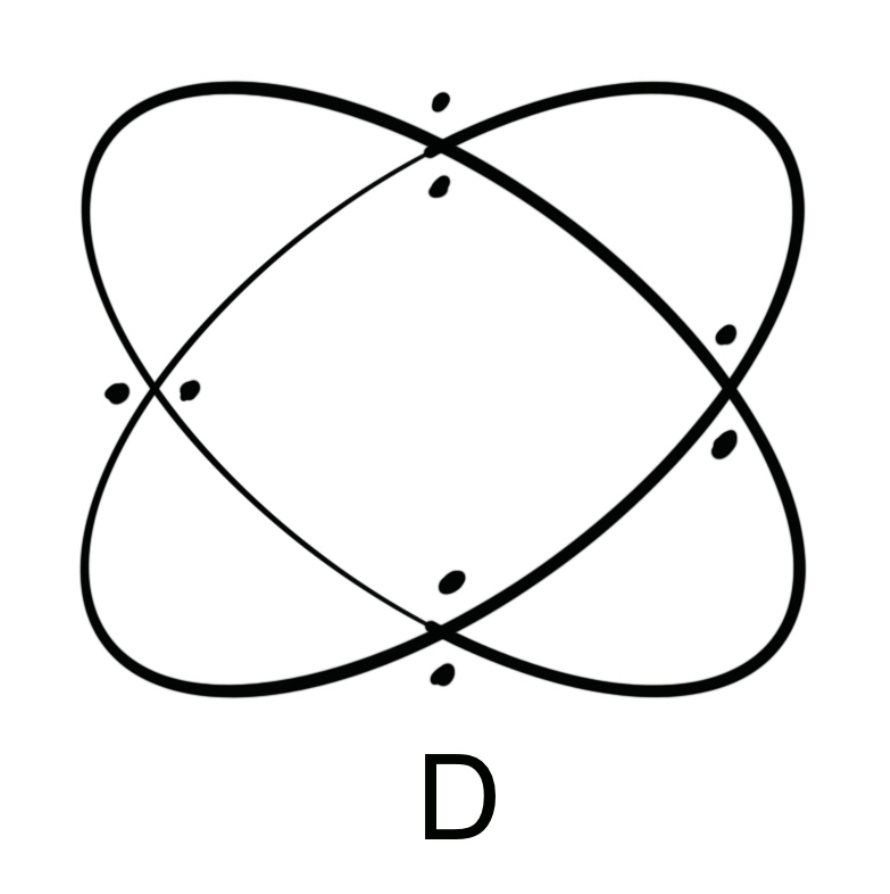}
\caption{Nonalternating link diagram with $S_0$ noncolorable} 
\label{cexample}
\end{center}
\end{figure}

\end{example}

Denote by $s(n,k)$ the Stirling symbol of the first kind. Its absolute value $(-1)^{n-k} s(n, k)$ is equal to the number of permutations of $n$ elements with exactly $k$ disjoint cycles. It is well known that 
$$(q)_n = \sum_{k=1}^{n} s(n,k) q^k.$$

\begin{theorem}\label{coefficients} The $k$th coefficient of $P(q)= \sum a_k q^k$  is 
$$a_k= \sum_m s(m, k)\frac{c_m}{m!}$$
\end{theorem}

\begin{proof} Beginning with equation (\ref {first}) we can write
$$P(q) = \sum_m (q)_m  {{c_m}\over m!}  = \sum_m (\sum_{k=1}^n s(m,k) q^k) \frac{c_m}{m!} = \sum_{k=1}^n (\sum_m s(m,k)\frac{c_m}{m!}) q^k$$
\end{proof}

Since $s(m,1) = (-1)^{m-1}(m-1)!$, the linear coefficient of $P(q)$ can be expressed as an exponential generating function of the $c_m$. 

\begin{cor} The linear term of $P(q) = \sum a_k q^k$ has coefficient $$a_1=\sum_m (-1)^{m-1} \frac{c_m}{m}$$ \end{cor}

We have already seen that the leading coefficient of $P(q)$ is $a_d=e_d$, the number of ways that we can color the associated link diagram $D$ with the maximum number $d$ of colors, up to renaming.  The second- and third-highest coefficients are given by the following.

\begin{cor} \label{penultimate} If $P(q)= \sum a_k q^k$ has degree $d$, then $$a_{d-1}=- \frac{d(d-1)}{2} e_d+e_{d-1}.$$ \end{cor}

\begin{cor} If $P(q)= \sum a_k q^k$ has degree $d$, then 
$$a_{d-2}=\frac{d(d-1)(d-2)(3d-1)}{24} e_d- \frac{(d-1)(d-2)}{2} e_{d-1} + e_{d-2}.$$ \end{cor}  

The number $e_m$  of Tait $m$-colorings of $(G,M)$ with exactly $m$ colors, up to renaming colors, can be computed from the values of the PK polynomial $P(q)$. 

\begin{prop} \begin{equation} \label{formula1} e_m= \frac{1}{m!} \sum_{j=0}^m (-1)^j {m \choose m-j} P(m-j).\end{equation}  \end{prop}

\begin{proof} Equation \ref{formula1} follows from an inclusion/exclusion argument. Consider $m$ tokens labeled $1, \ldots, n$. Let  $P(j)$ be the ``cost" of selecting any collection of $j$ tokens, where $j \le m$. Then the cost of selecting exactly $m$ tokens is given by
$$P(m) - {m \choose m-1} P(m-1) +{m\choose m-2}P(m-2)- \cdots \pm {m\choose m}P(0).$$
Interpret the cost of any collection of $m$ tokens as the number of ways to $m$-color $(G,M)$.  Then $e_m$ is obtained by dividing by $m!$. \end{proof}

\begin{prop} \label{linearcoef} Assume that $D$ is a semi-reduced alternating link diagram in the plane. If $P(q) = \sum a_k q^k$ is its PK polynomial, then $a_1$ is even. \end{prop}

\begin{proof}  First observe that the linear coefficient of $(q)_2$ is $-1$, while for $m>2$ the linear coefficient of $(q)_m$ is even. Consequently, the contribution to $a_1$ from $m$-colorings of $D$ with $m>2$ is even. It suffices to prove that this is also the case for contributions from $2$-colorings.

We follow the idea of the proof of Corollary 7 in \cite{aigner}.  If the dotted regions of $D$ each have an even number of arcs (equivalently, the Tait graph is Eulerian), then color the arcs of each dotted region alternately with two colors. We can smooth a subset of the crossings in such a way that that the assignment is a 2-coloring of a state of $D$. All 2-colorings of states arise this way.  Hence the number of 2-colorings of states of $D$ is $2^V$, where $V$ is the number of dotted regions (equivalently, the number of vertices of the Tait graph). In this case, the coefficient $e_2$ of the falling factorial expression of $P(q)$ is even, and the contribution to $a_1$ from 2-colorings is even. 

If some dotted region of $D$ has an odd number of arcs,  then no state of $D$ has a 2-coloring, and the contribution to $a_1$ from $2$-colorings is zero. 
\end{proof} 

\begin{remark} Figure \ref{nonaltex}  shows that the requirement of Proposition \ref{linearcoef} that the link diagram be alternating is needed. \end{remark} 

\begin{theorem} \cite{aigner} Let $D$ be a semi-reduced alternating link diagram in the plane with PK polynomial $P(q)=\Sigma_{k=1}^n a_kq^k$.  Assume that the Tait graph of $D$ is Eulerian. Then the coefficients of $P(q)$ other than the constant term are all nonzero and alternate in sign.  \end{theorem}

\begin{proof} We will prove the result for connected diagrams $D$.  This will suffice since the PK polynomial of a disjoint union of $r$ link diagrams is the product of the associated Penrose polynomials divided by $q^r$.

The hypothesis assures that we can orient boundaries of the all-smoothed state so that any pair of arcs that run parallel to each other, as in Figure \ref{eulerian}, have the same directions.  One checks easily that each time we add back a crossing to a state, we increase or decrease its number of components by one.

\begin{figure}[H]
\begin{center}
\includegraphics[height=1.5 in]{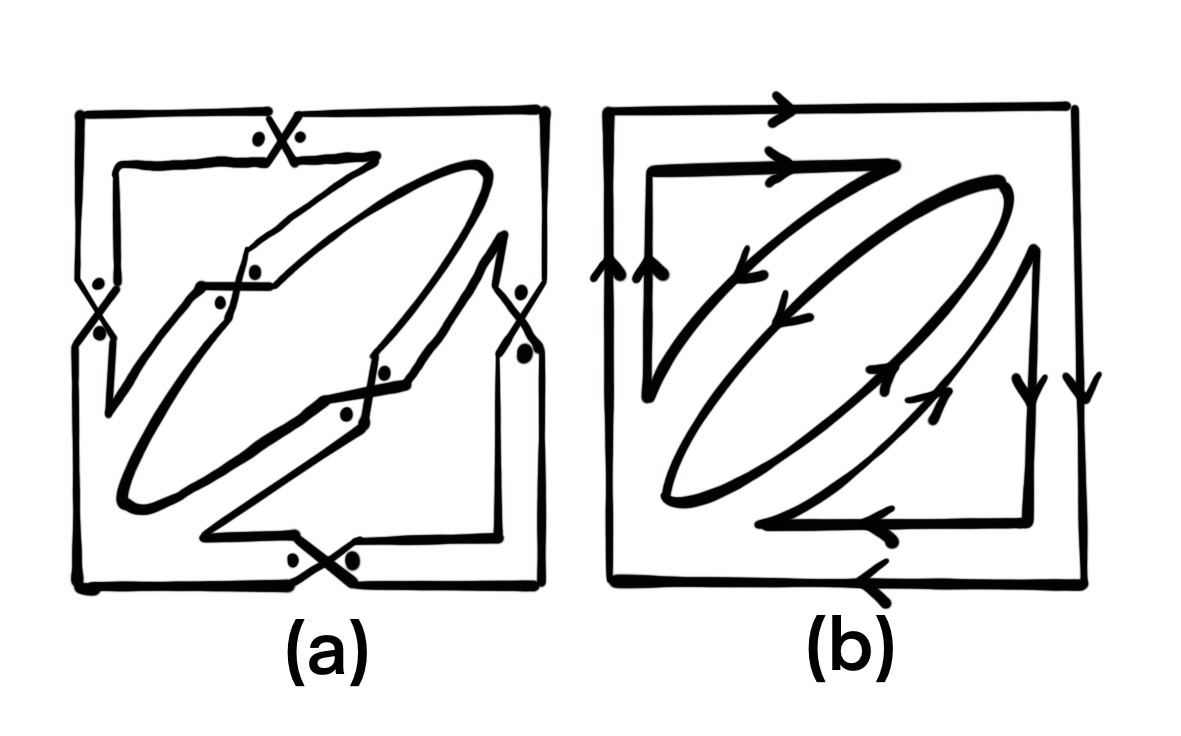}
\caption{(a) link diagram $D$; (b) orientation of all-smoothed state $S_0$ }
\label{eulerian}
\end{center}
\end{figure}

The colorable states must each have an even number of crossings since two components intersect in an even number of points and no component of a colorable state crosses itself. Thus the number of components of a colorable state must have the same parity as $||S_0||$.  Since $D$ is connected, so is each component graph $G_S$. By well-known properties of the chromatic polynomial, the degree of $\chi(G_S)$ is the number of components of $S$ and its leading coefficient is positive; the coefficients of $\chi(G_S)$ other than the constant coefficient are nonzero and alternate in sign (see \cite{biggs}). Since the PK polynomial of $D$ is the sum of such polynomials, its coefficients satisfy the same condition. 

\end{proof} 

\begin{remark}  If the Tait graph of $D$ is not Eulerian, then the PK polynomial of $D$ need not be a sum of chromatic polynomials with degrees of the same parity. For example,  the Tait graph of the usual diagram of the Borromean rings is the complete graph $K_4$ on four vertices. Its PK polynomial is the sum of the chromatic polynomials of $K_3$ (odd degree) and $K_4$ (even degree). However, the PK polynomial has coefficients that are sign-alternating. 
\end{remark}

\ni {\bf Conjectures} (see \cite{aigner}):  Let $(G,M)$ be a connected cubic plane graph with perfect matching. 

(1) The coefficient of the linear term of $P(q)$ is non-zero.   

(2) The coefficients of $P(q)$ are ``weakly sign-alternating," that is, alternately non-negative and non-positive.  \bs

Conjecture (1) is false without the hypothesis that $G$ (or the associated link diagram) is planar. A counterexample appears in Figure \ref{vbrings} below. The link diagram has exactly three colorable states. One of them is the diagram itself. The others each have two intersecting components. The polynomial $P(q)$ 
is $q(q-1)(q-2) + 2 q(q-1) = q^3-q^2.$

\begin{figure}[H]
\begin{center}
\includegraphics[height=2.5 in]{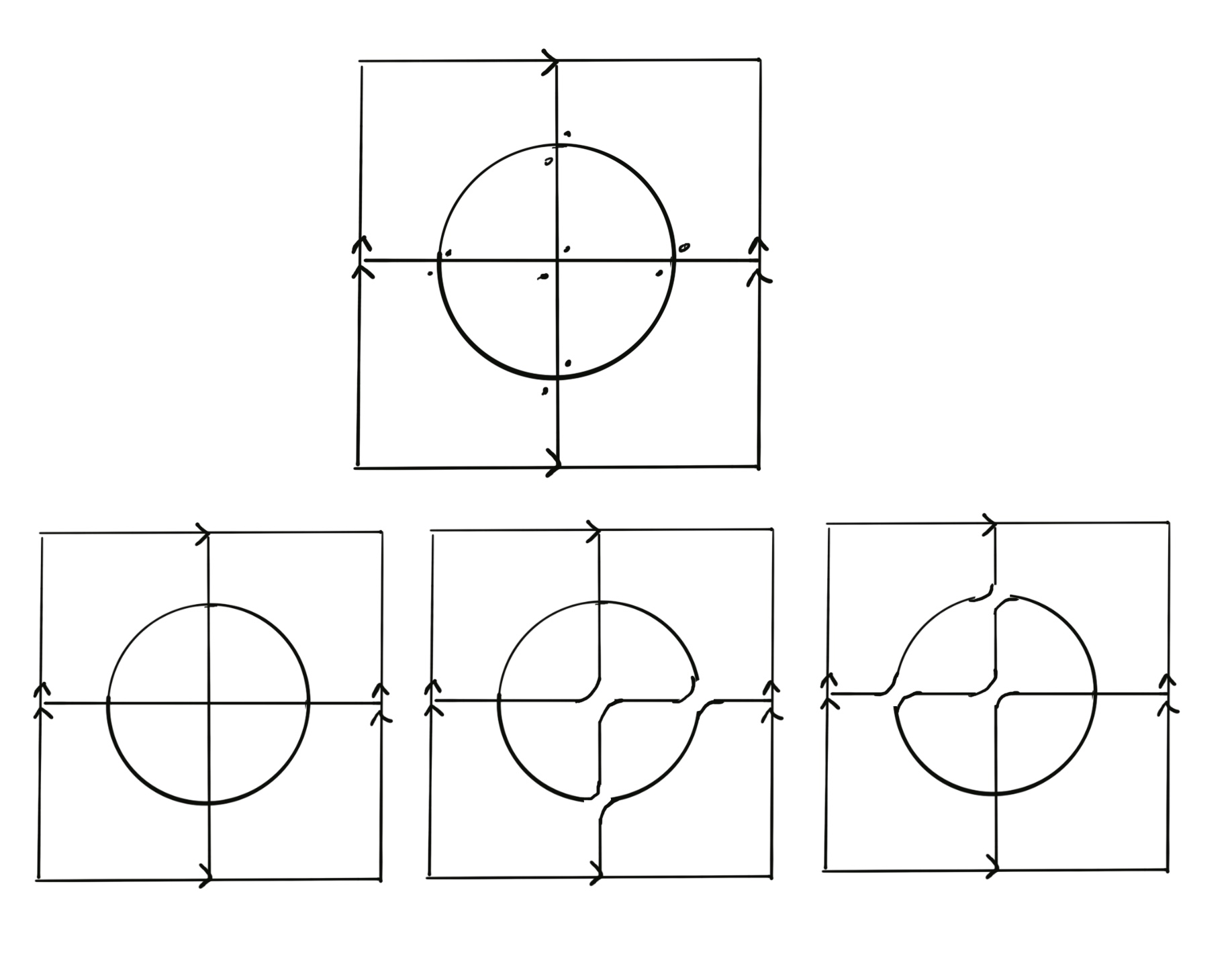}
\caption{Link diagram in torus and colorable states} 
\label{vbrings}
\end{center}
\end{figure}

If Conjecture (1) is false, then Proposition \ref{Prime} implies that for arbitrarily large $s$ that there exist $(G,M)$ such that the coefficients $a_k$ of $P(q)$ satisfy $a_k =0$ for all $k \le s$.

The value of $P(q)$ at negative integers can also be computed using the description of $P(q)$ in terms of falling factorials. 

\begin{prop} $$P(-n) = \sum_m(-1)^m {n+m-1 \choose m} c_m.$$ \end{prop}\

\begin{proof} Begin with: 
$$(-n)_m = (-n)(-n-1) \cdots (-n-m+1)$$
$$=(-1)^m n(n+1) \cdots (n+m-1)$$
$$=(-1)^m \frac{(n+m-1)!}{(n-1)!}.$$
From equation (\ref{first}) with $q=-n$: 
$$P(-n) = \sum_m (-n)_m \frac{c_m}{m!}.$$
Substituting the expression for $(-n)_m$, we have: 
$$P(-n) = \sum_m (-1)^m \frac{(n+m-1)!}{(n-1)!} \frac{c_m}{m!}$$
$$=\sum_m(-1)^m {n+m-1 \choose m} c_m.$$

\end{proof}

\begin{cor} $P(-1) = \sum_m (-1)^m c_m$ and $P(-2) = \sum_m (-1)^m (m+1) c_m$ \end{cor}

\begin{remark}  If Conjecture (2) is true, then $|P(-1)|$ is the sum $\sum_m |a_j|$, the 1-norm of $P(q) =  \sum a_j q^j$. \\
\end{remark}

\begin{remark} In \cite{jaeger} Jaeger proved that $P(-2)$ is equal to a nonzero multiple of the number of proper edge 3-colorings of $G$. For reduced alternating link diagrams $D$ the PK polynomial $P(q)$ would be nonzero and sign-alternating, and that would imply that $P(-2)$ is nonzero.  Using Theorem \ref{equivalent}, a proof of Conjecture 2 would yield a proof of the Four Color Theorem.   \end{remark}


\section*{Appendix: Kauffman's Formulation }\label{appendix}

Roger Penrose \cite{Penrose} gave a graphical recursion formula for calculating the number of Tait 3-colorings of a planar cubic graph G (see Section \ref{TaitColorings}).
 The formula is summarized as shown in Figure~\ref{Figure1}. In this figure it is understood that any loop has value 3, even if the loop is self-intersecting.
 The theory of the Penrose formula is based on tensor networks and can be seen in \cite{Penrose, KP}. In fact, the basic structure behind the Penrose formula depends on a simple combinatorial tautology shown in 
 in  Figure~\ref{Figure2}.  Here $\{G\}$ denotes the set of colorings of $G$ and the wiggly line denotes the assertion that the arcs connected by the wiggly line are colored differently.
 Thus the possibilities at an edge are that two colors distinct from the edge color are either parallel or permuted as is shown in  Figure~\ref{Figure2}.  Thus the equation shown in  Figure~\ref{Figure2} is a logical tautology.\\
 
 \begin{figure}
     \begin{center}
     \begin{tabular}{c}
     \includegraphics[width=8cm]{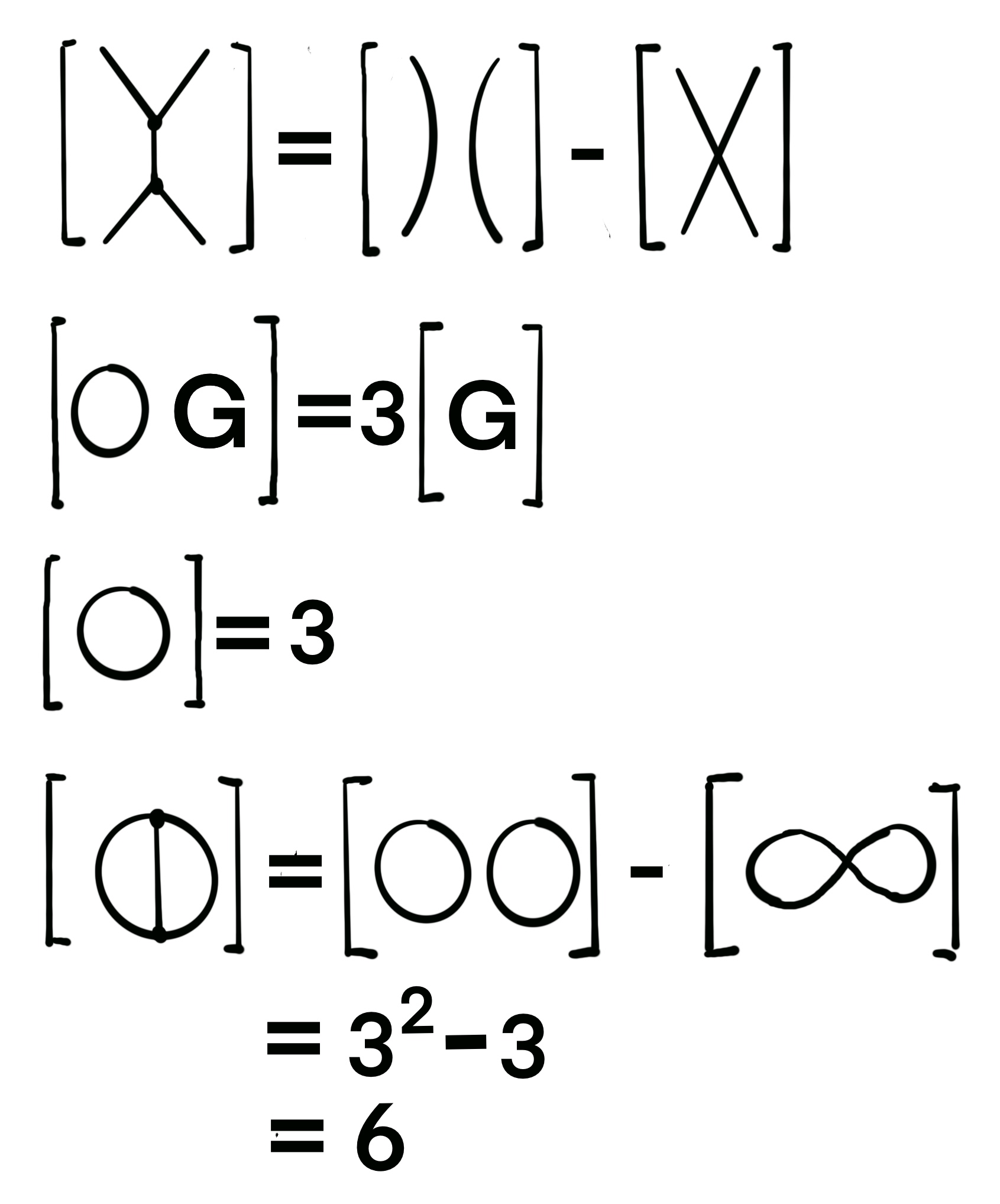}
     \end{tabular}
     \caption{Penrose Formula}
     \label{Figure1}
\end{center}
\end{figure}

\begin{figure}
     \begin{center}
     \begin{tabular}{c}
     \includegraphics[width=10cm]{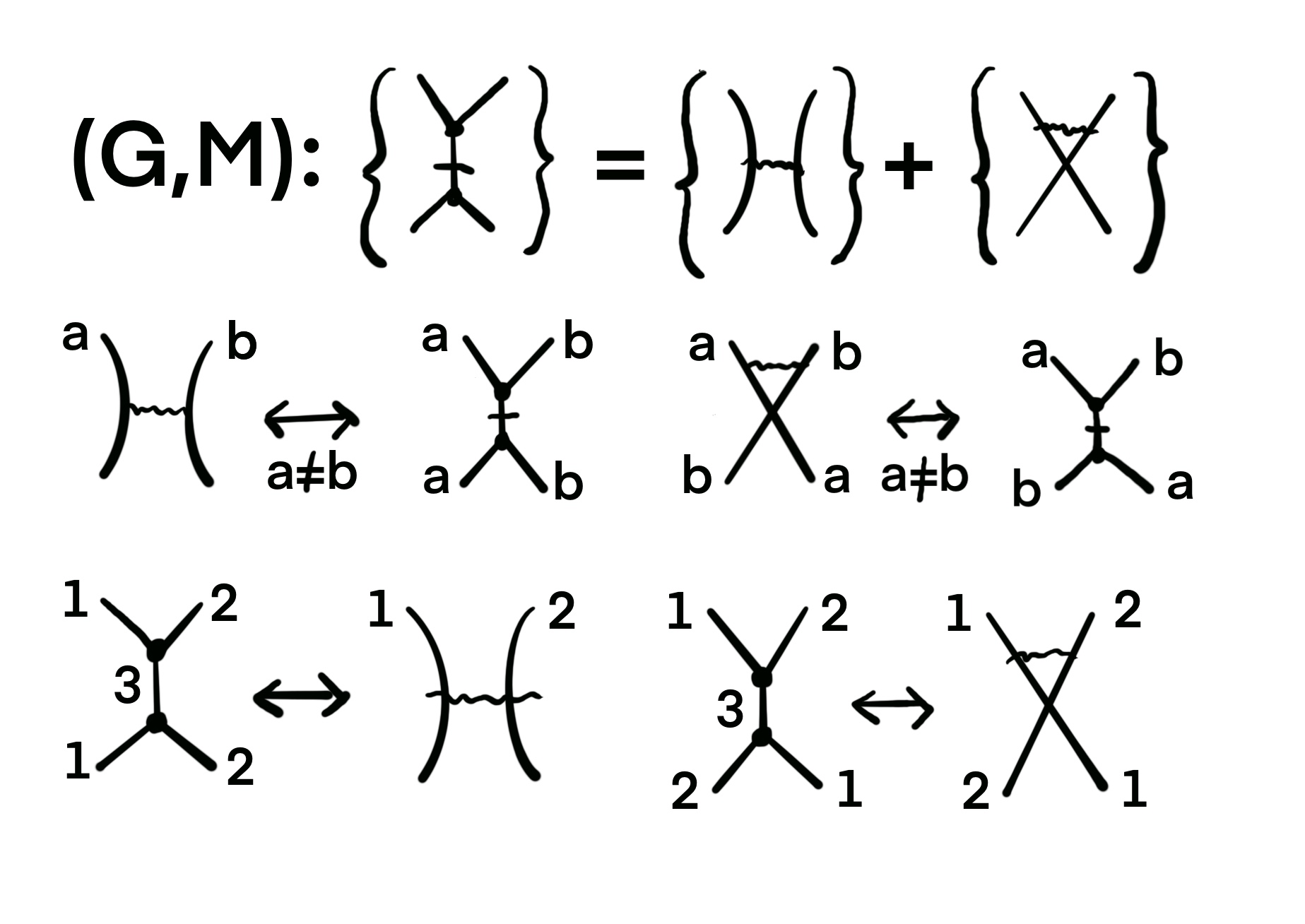}
     \end{tabular}
     \caption{Tautological Formula}
     \label{Figure2}
\end{center}
\end{figure}

 The tautology indicates how to generalize this notion of coloring a graph with 3 colors to colorings with $n$ colors for a graph $G$ with a given perfect matching $M.$ Here $M$ denotes a set of disjoint ``matching edges" 
 whose endpoints comprise the entire set of endpoints of the graph $G.$ Given such a pair $(G,M)$ one can try to color the edges of $G \setminus M$ with $n$ colors so that at a matching edge there are exactly two colors, distinct at each node incident to the matching edge. This means that at one node there are two colors and at the other node, there are the same two colors, parallel or permuted. See  Figure~\ref{Figure2} again. In that figure the mark on an edge means that it is 
 a matching edge. Note that in this way of coloring the matching edges do not receive a color. In the case of $n=3$ the matching edge does receive a color, the remaining third color not used. \\

 As the reader can see, the color expansion $\{G,M\}$ can be written as sum over all the ways of replacing each matching edge with either parallel or crossed arcs. See the example in  Figure~\ref{Figure3}.
 Each loop configuration, in this sum for the general color expansion, presents its own coloring problem, and can be reformulated as a standard graph coloring problem. To accomplish the reformulation, 
 replace each loop in the sum by a node and replace each wiggly line by a graphical edge between such nodes. Call this new graph $\G[S]$ where $S$ denotes the given loop configuration. (Configurations $S$ such that $G[S]$ has no vertices with self-loops correspond to the colorable states defined above.)
 See  Figure~\ref{Figure3} for an example. To count the number of colorings of a given loop configuration $S$ we can compute the chromatic polynomial $\chi(\G[S],n)$ of the graph associated with $S.$
 Then $$P(G,M)(n) = \sum_{S} \chi(\G[S],n)$$  counts the colorings of $(G,M)$ when $S$ runs over the loop configurations.
 
 \begin{figure}
     \begin{center}
     \begin{tabular}{c}
     \includegraphics[width=8cm]{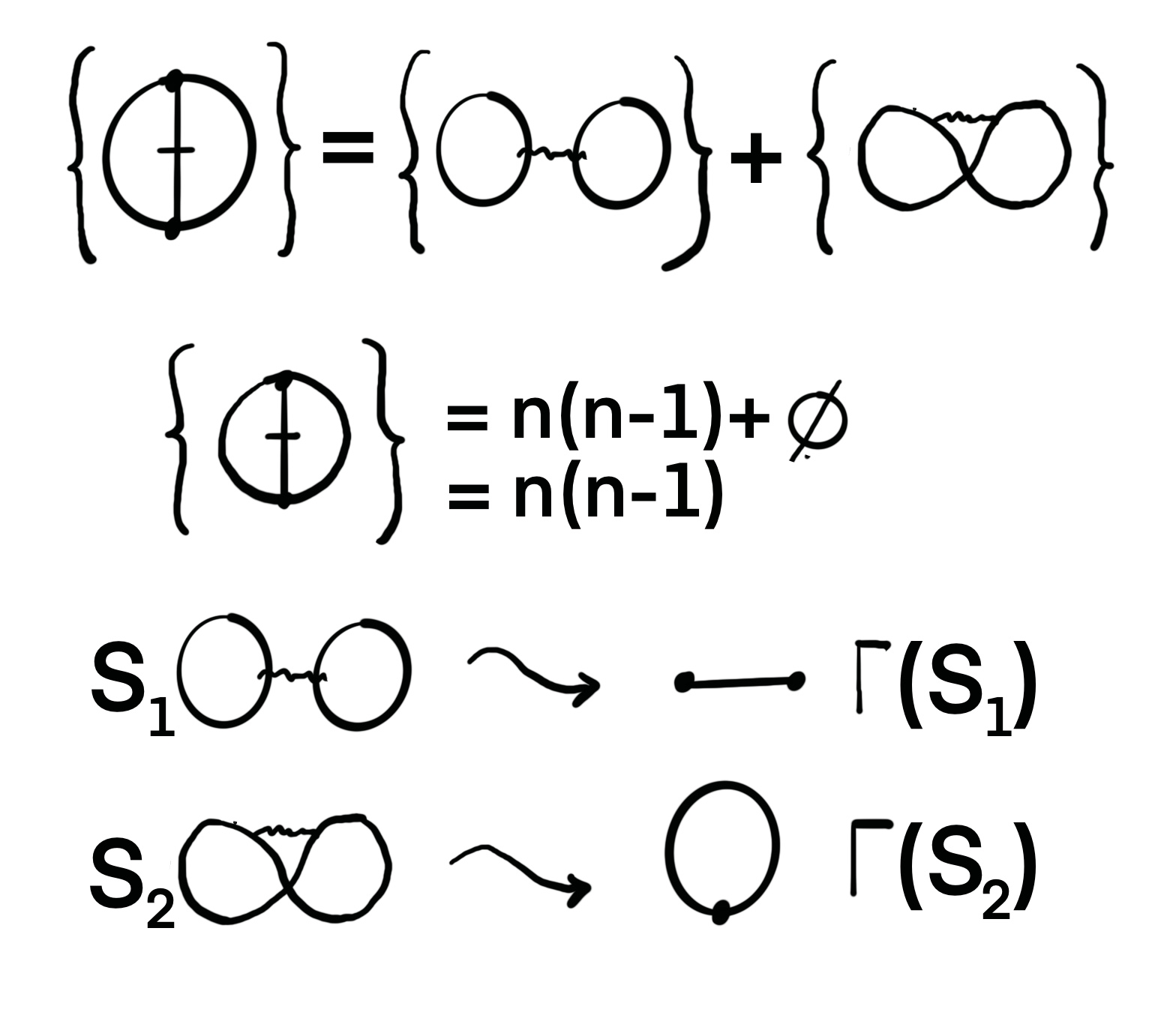}
     \end{tabular}
     \caption{Graphs for States}
     \label{Figure3}
\end{center}
\end{figure}

 We can formulate this use of the generalized tautological expansion entirely in terms of loop states by noting the equations in  Figure~\ref{Figure4} where the nodal crossed arcs mean that there is only one color assigned to the arcs incident to the node. The diagrammatic equations of  Figure~\ref{Figure4} correspond to the logical identity {\it Different = Anything - Same}. Then we have an expansion formula for counting Tait $n$-colorings of perfect matchings
in terms of the given diagrams related to the graph.  This formula can be taken in the spirit of the original Penrose formula and it applies to any cubic graph, not just to graphs that are embedded in the plane.\\
 
 \begin{figure}
     \begin{center}
     \begin{tabular}{c}
     \includegraphics[width=8cm]{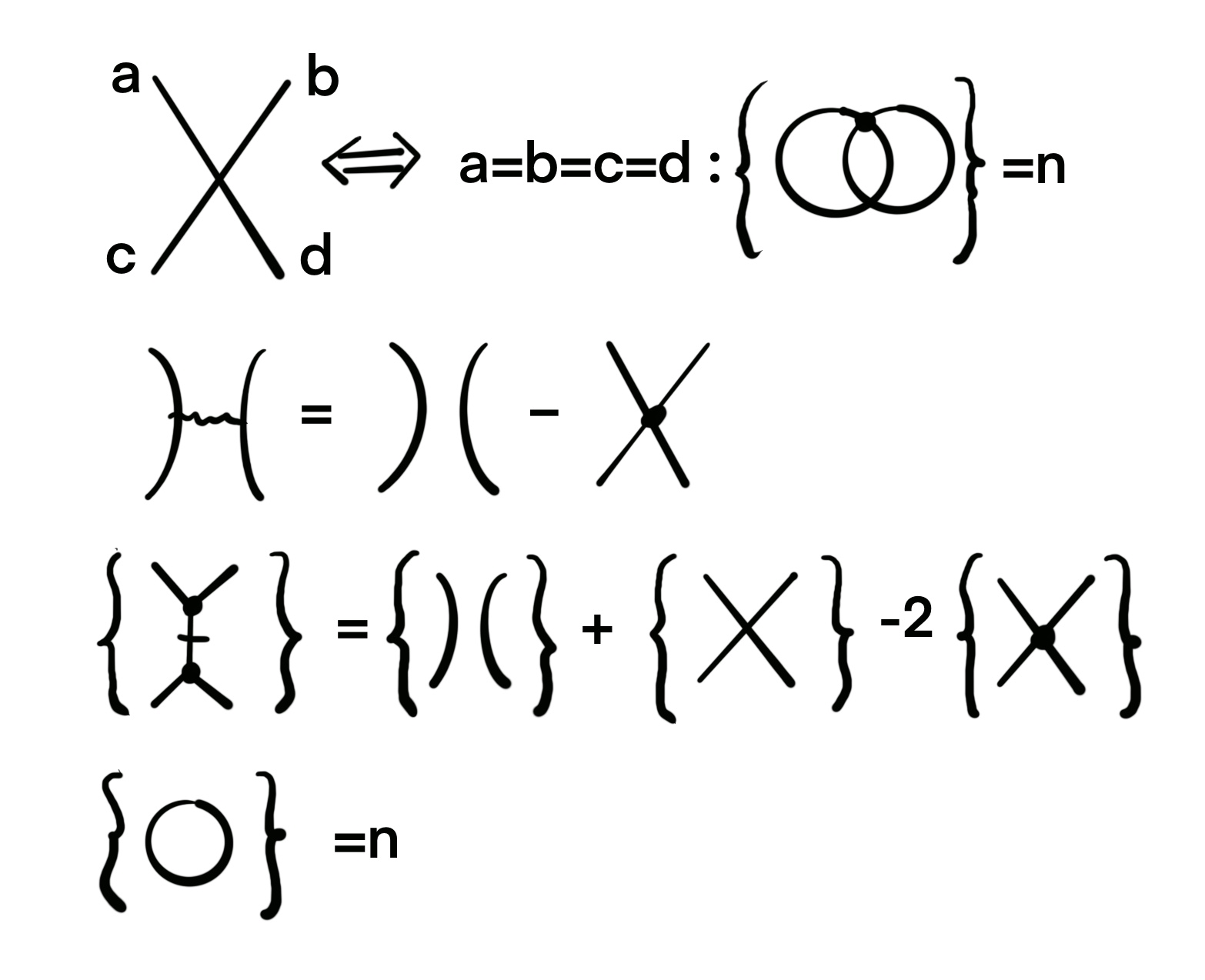}
     \end{tabular}
     \caption{Generalized Penrose Expansion}
     \label{Figure4}
\end{center}
\end{figure}

 By replacing the matching edge by a knot-theoretic crossing as in  Figure~\ref{Figure5}, we can formulate all these considerations in knot and link diagrams. In this figure we use a virtual crossing designation on the crossed arcs.
 One can start with virtual diagrams when working with non-planar graphs. The virtual crossings are then crossings resulting from taking representative immersions of the graphs in the plane.\\
 
 \begin{figure}
     \begin{center}
     \begin{tabular}{c}
     \includegraphics[width=6cm]{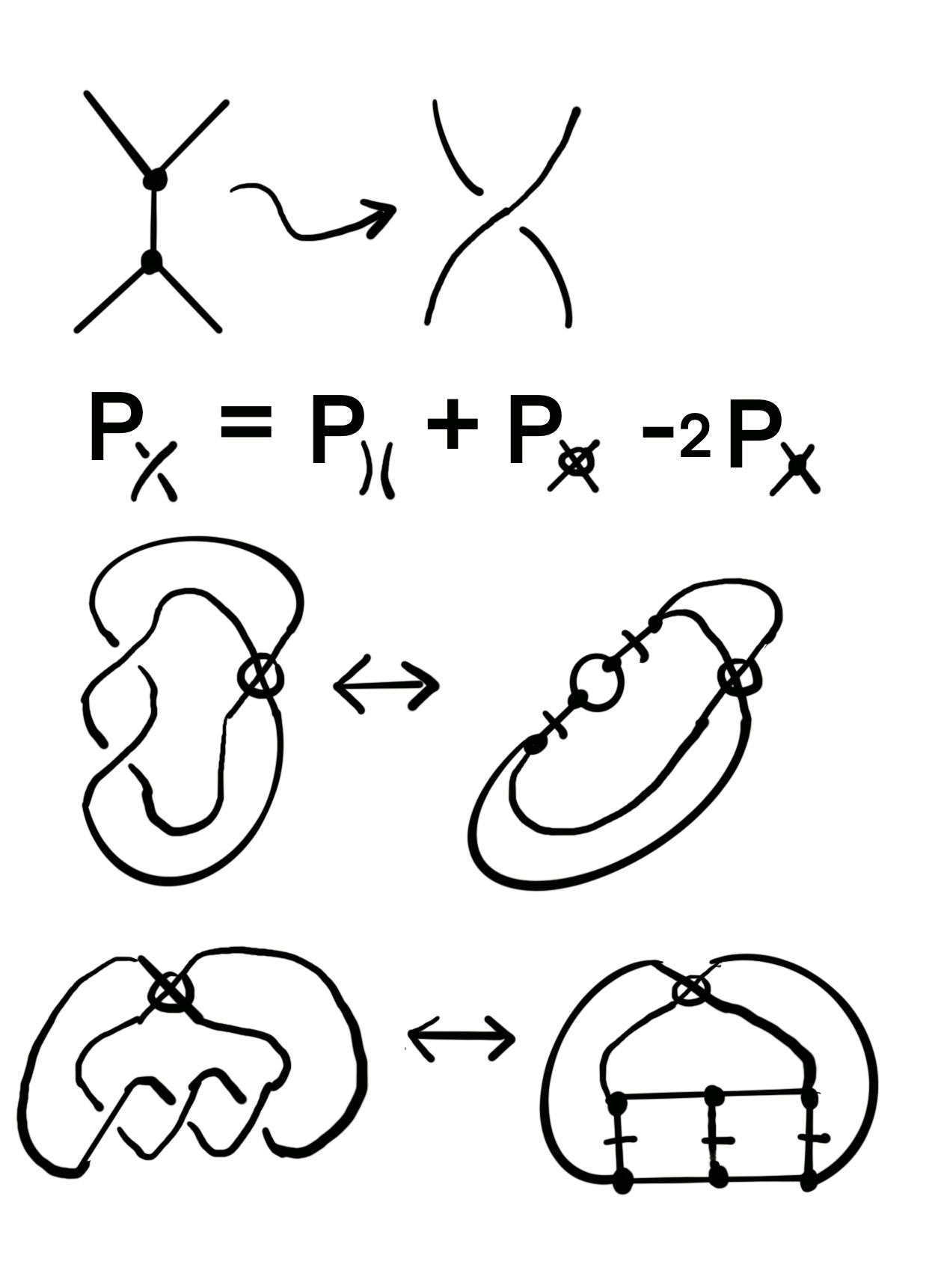}
     \end{tabular}
     \caption{Knot Diagram Expansion}
     \label{Figure5}
\end{center}
\end{figure}

It is worth mentioning that the logical tautology at $n=3$ gives a criterion for Tait colorability. Any trivalent graph can be arranged in the form shown in  Figure~\ref{Figure6} where $B$ is a black box containing the rest of the 
graph. Then we have the formula shown in  Figure~\ref{Figure6}. This means that if the graph $G$ is a minimal non-colorable graph, then the colorings of the two smaller graphs with wiggly lines must have the same colors at
the wiggles. That is the two closures of the black box both force same colors for the two closure arcs.  The simplest example of such a forcing box is also shown in Figure~\ref{Figure6} as a crossover of two arcs. 
In this case the uncolorable graph $G$ is equivalent to the planar dumbell graph shown in that figure, a graph with an isthmus. In  Figure~\ref{Figure7} we show how a more sophisticated black box produces the 
uncolorable and non-planar Petersen graph. The four-color theorem asserts that forcing black boxes that produce 1-connected graphs $G$ will have $G$ non-planar.\\

\begin{figure}
     \begin{center}
     \begin{tabular}{c}
     \includegraphics[width=10cm]{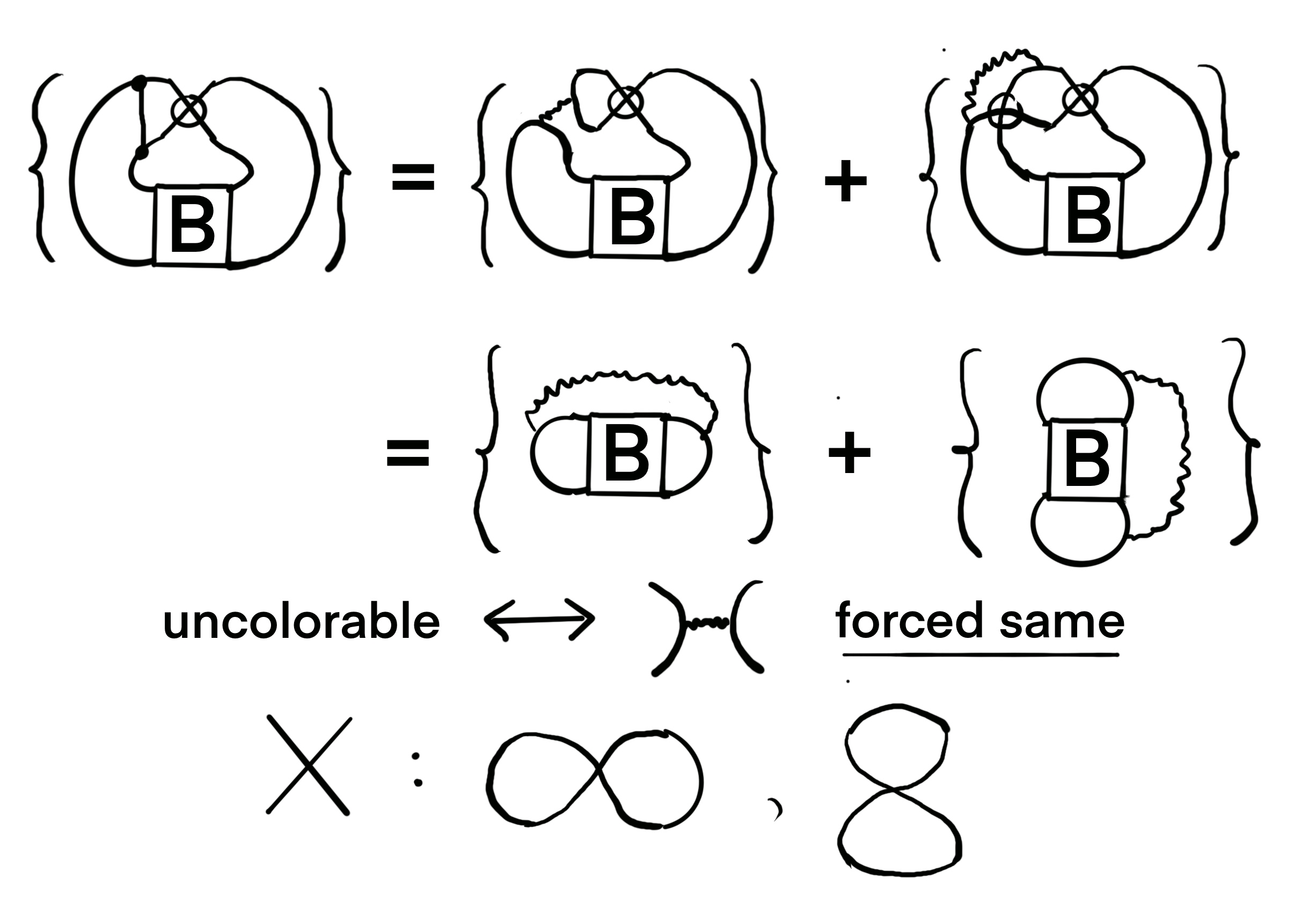}
     \end{tabular}
     \caption{Uncolorability Criterion}
     \label{Figure6}
\end{center}
\end{figure}

\begin{figure}
     \begin{center}
     \begin{tabular}{c}
     \includegraphics[width=8cm]{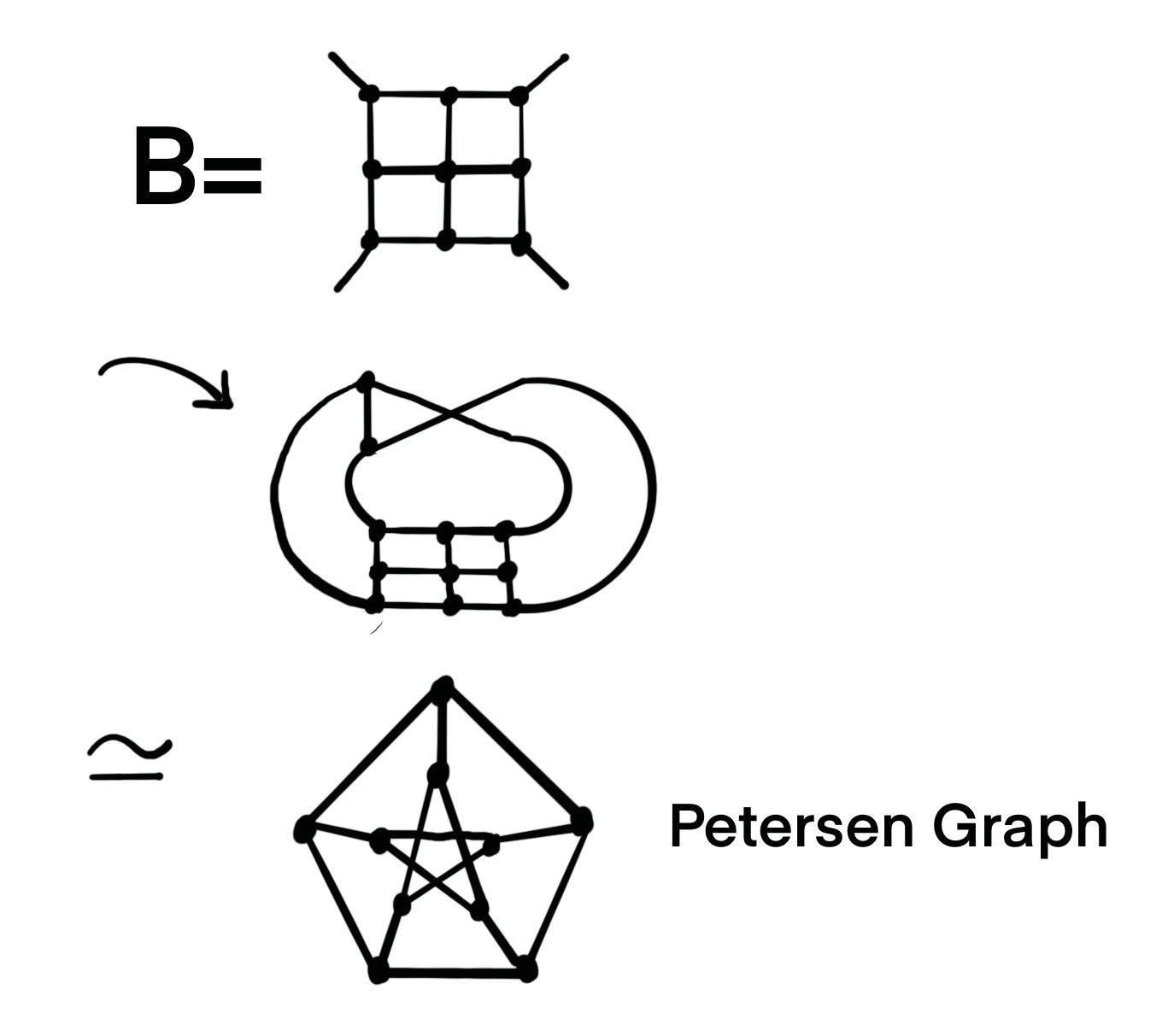}
     \end{tabular}
     \caption{Uncolorability Criterion and Petersen Graph}
     \label{Figure7}
\end{center}
\end{figure}

\newpage


 \end{document}